\documentclass{article}
\usepackage{fullpage} 
\usepackage{ulem}
\usepackage{amsmath}   
\usepackage{amssymb}   
\usepackage{amsthm}    
\usepackage{amsfonts}
\usepackage{physics}
\usepackage[capitalise]{cleveref}
\usepackage{bbm}
\usepackage[T1]{fontenc}   
\usepackage{algpseudocode}
\usepackage{tikz}
\usepackage{amsmath}
\usepackage{pgf,tikz}
\usepackage{graphicx} 
\graphicspath{{./}{figures/}} 
\usetikzlibrary{arrows}
\usepackage{amsopn}
\usepackage{adjustbox}

\def \to{\rightarrow}
\def \states{\mathbb{T}^d}
\def \T{\mathbb{T}}
\def \R {\mathbb{R}}
\def \N {\mathbb{N}}

\def \mes{\mathcal{P}}

\def \mint{\int_{\states}}

\newcommand{\be}{\begin{equation}}
\newcommand{\ee}{\end{equation}}

\theoremstyle{definition}

\theoremstyle{plain}
\newtheorem{theorem}{Theorem}[section]
\newtheorem{proposition}[theorem]{Proposition}

\newtheorem{lemma}[theorem]{Lemma}

\theoremstyle{remark}
\newtheorem{remark}[theorem]{Remark}

\numberwithin{equation}{section}

\begin{document}

\title{On the quadratic convergence of Newton's method for  Mean Field Games with non-separable Hamiltonian}
\date{}
\author{Fabio Camilli \and Qing Tang}

\maketitle

\begin{abstract}
We analyze  asymptotic convergence properties of Newton's method for a class of evolutive Mean Field Games systems with non-separable Hamiltonian arising in mean field type models with congestion. We prove the well posedness of the Mean Field Game system with non-separable Hamiltonian and of the linear system giving the Newton iterations. Then, by forward induction and assuming that the initial guess is sufficiently close to the solution of problem, we show a quadratic rate of convergence  for the approximation of the Mean Field Game system by Newton's method. We also consider the case of a nonlocal coupling, but with separable Hamiltonian, and we show a similar rate of convergence.
\end{abstract}
\noindent
{\bf AMS Subject Classification:} 49N70, 91A13, 35Q80, 65M12\\
{\bf Keywords:} {~Mean Field Games, non-separable Hamiltonian, Newton method, congestion model, numerical methods}

\section{Introduction}
Mean Field Games (MFG in short)  theory, introduced in \cite{ll,hcm}, arises in the study of differential games with an infinite number of rational agents. The corresponding literature  is now vast and concerns both theoretical and applicative aspects, see \cite{achdouCetraro,carmona2018,Lions} and references therein. In this regard, a  significant part of it is dedicated to the study of numerical methods and algorithms for the computation of the solution  to the MFG model, both in the formulation as a PDEs system   and as an optimal control problem of a PDE. Such   approaches, just to mention a few, include finite differences, semi-Lagrangian methods and Fourier expansions with regard to the approximation methods (see  \cite{ad,acd,achdou2012,achdou2015,blp,ccg,cs, lauriere2021,lst,li2021simple,ns,tang2022learning}).   Many of these methods exploit  the variational structure of the problem, concerns the case in which the coupling term involving the distribution of the population   is separated from the Hamiltonian, while relatively few works have been dedicated to the so-called non-separable case
\be\label{MFG}
\left\{\begin{aligned}
	(i) \qquad &-\partial_t u -\Delta u+H(x,m,Du)=f(m) \qquad &&{\rm in}\,\,Q,\\
	(ii) \qquad &\partial_t m-  \Delta m  - \text{div} \big( mH_p(x,m,Du)\big) =0 &&{\rm in}\,\,Q, \\
	&m(x,0)=m_0(x) , \; u(x,T)=u_T(x) &&{\rm in}\,\,\T^d. 
\end{aligned}\right.
\ee
($\T^d$ is the unit torus and $Q=\T^d\times (0,T)$). 
Moreover, the non-separable case is very important in applications   to model congestion effects, i.e. situations in which the cost of displacement of the agents increases in those regions where the density is large. MFGs models including congestion were introduced   in \cite{Lions} and   a typical Hamiltonian in such cases is
\begin{equation*} 
	H(x,m,p)=\frac{h(x)|p|^2}{(1+m)^\alpha}, \quad \alpha>0.
\end{equation*}
Global in time weak solution to \cref{MFG} has been considered  in \cite{achdou2018,ghattassi2023non},
 short time existence and uniqueness of regular solution     in \cite{cirant2020,ambrose2023well} and the stationary case in \cite{gomes}. In general, MFGs with non-separable Hamiltonian do not have a variational structure and this causes a restriction on the choice of numerical methods. Moreover and in general, implicit schemes are preferred as they enhance the stability and efficiency compared to explicit schemes. To design implicit finite difference schemes, iterative methods are needed to reduce the problem to a sequence of linear systems. Iterative methods employed in solving MFGs include Newton's method \cite{achdou2012,achdou2015,berry2024approximation,lauriere2021}, fixed point iteration, fictitious play, policy iteration \cite{lst}, smoothed policy iteration \cite{tang2022learning} etc. In particular, numerical solution of MFGs with non-separable Hamiltonian  have been discussed  in e.g. \cite{achdou2012,achdou2015,lauriere2021,lst,ghattassi2023non}.\par
 In this paper, we consider  Newton's method  from a continuous standpoint, viewing it as a linear system of partial differential equations (PDEs) which approximate the nonlinear problem \eqref{MFG}.   Newton's method (also known as the Newton Kantorovich method) is effective for convex optimization problems (\cite{boyd2004}) or for solving nonlinear functional equations in Banach spaces, cf. \cite{ciarlet2012newton}. The novelty and main contributions of this work is theoretical.  We rigorously establish a quadratic rate of convergence of the   method in a neighborhood of the solution of \cref{MFG}. 
 In the study of Newton's method to \cref{MFG}, a critical point is in establishing the well-posedness of the linearized MFG system. To address this, we broaden the theoretical framework developed in \cite{cardaliaguet2019,bc} for analyzing master equations. This extension is applicable to MFGs with non-separable Hamiltonians, subject to certain Hessian-type monotonicity conditions. Recently, the convergence analysis of Newton's method has been considered  in \cite{berry2024approximation} for stationary MFGs with separable Hamiltonian. 
\par
The Newton  method for  system \cref{MFG} reads as follows. Writing the MFG system as an operator equation $\mathcal{F}(u,m)=0$ and denoting by $\mathcal{LF}_{(\check{u},\check{m})}$ the linearized $\mathcal{F}$ at $(\check{u},\check{m})$ and by $\mathcal{LF}_{(\check{u},\check{m})}^{-1}$ the inverse of $\mathcal{LF}_{(\check{u},\check{m})}$, then we get
\begin{equation*}
(u^{n},m^n)=(u^{n-1},m^{n-1})-\mathcal{LF}^{-1}_{(u^{n-1},m^{n-1})}\cdot\mathcal{F}(u^{n-1},m^{n-1}),
\end{equation*}
or equivalently 
\begin{equation}\label{LF}
\mathcal{LF}_{(u^{n-1},m^{n-1})}\cdot \big(u^{n}- u^{n-1},m^n-m^{n-1}\big)=-\mathcal{F}(u^{n-1},m^{n-1}).
\end{equation}
The previous identity in PDE form reads as
\be
\left\{\begin{aligned}
(i)\qquad &-\partial_t (u^{n}- u^{n-1}) -\Delta (u^{n}- u^{n-1})+H_p(x,m^{n-1},Du^{n-1})D(u^{n}- u^{n-1})\\
&+H_m(x,m^{n-1},Du^{n-1}) (m^{n}- m^{n-1})-f'(m^{n-1})(m^n-m^{n-1})\\
={}&\partial_t u^{n-1} +\Delta u^{n-1}-H(x,m^{n-1},Du^{n-1})+f(m^{n-1})\qquad&&{\rm in}\,\,Q, \\
(ii)\qquad &\partial_t (m^{n}- m^{n-1})- \Delta (m^{n}- m^{n-1}) - \text{div} \big( (m^{n}- m^{n-1}) H_p(x,m^{n-1},Du^{n-1})\big)\\
&-\text{div} \Big(m^{n-1} H_{pm}(x,m^{n-1},Du^{n-1})(m^{n}- m^{n-1})\Big) \\
&-{\rm{div}}\Big(m^{n-1}H_{pp}(x,m^{n-1},Du^{n-1})(Du^n-Du^{n-1})\Big)\\
={}&-\partial_t m^{n-1}+  \Delta m^{n-1} +\text{div} \big( m^{n-1}H_p(x,m^{n-1},Du^{n-1})\big) &&{\rm in}\,\,Q,    \\
&m^n(x,0)=m_0(x) , \; u^n(x,T)=u_T(x) &&{\rm in}\,\,\T^d,
\end{aligned}\right.
\ee
and, after simplification, we get the coupled linear system in the unknown $(u^{n},m^{n})$
\be\label{Newton system}
\left\{\begin{aligned}
(i)\qquad &-\partial_t u^n -\Delta u^n+H_p(x,m^{n-1},Du^{n-1})D(u^{n}- u^{n-1})+H_m(x,m^{n-1},Du^{n-1}) (m^{n}- m^{n-1})\\
={}&-H(x,m^{n-1},Du^{n-1})+f(m^{n-1})+f'(m^{n-1})(m^n-m^{n-1})\qquad&&{\rm in}\,\,Q,  \\
(ii)\qquad &\partial_t m^n- \Delta m^n  - \text{div} \big( m^nH_p(x,m^{n-1},Du^{n-1})\big) \\
={}&{\rm{div}}\Big(m^{n-1}H_{pp}(x,m^{n-1},Du^{n-1})(Du^n-Du^{n-1})\Big)\\
&+\text{div} \Big(m^{n-1} H_{pm}(x,m^{n-1},Du^{n-1})(m^{n}- m^{n-1})\Big)&&{\rm in}\,\,Q,    \\
&m^n(x,0)=m_0(x) , \; u^n(x,T)=u_T(x) &&{\rm in}\,\,\T^d.
\end{aligned}\right.
\ee
Assuming that the Hamiltonian is regular and satisfies a classical monotonicity condition (see \cite{Lions,achdou2018})   we obtain existence and  uniqueness of a classical solution $(u,m)$ to \cref{MFG}, see \cref{Well posed}. Then we  prove the well posedness of  \cref{Newton system} at each iteration and the   quadratic rate of convergence of the Newton iteration    to   the solution of the MFG system when the initial guess $(u^0,m^0)$ is sufficiently close to $(u,m)$.  We remark that, even though the algorithm is presented for evolutive MFGs, the ideas extend naturally to stationary MFGs as well. \par

This paper primarily focuses on analyzing the convergence of the Newton method for the MFG system at the level of PDEs.  In the MFG literature, this iterative  method has been applied to solve the nonlinear finite dimensional system
which results via a finite differences approximation of \cref{MFG}, see  \cite{achdou2012,achdouCetraro}.	 The   algorithm is usually coupled with a continuation method (typically with respect to the viscosity parameter). Indeed,  it is important to have a good initial guess of the solution and, for that, it is possible to take advantage of the continuation method by choosing the initial guess as the solution obtained with the previous value of the parameter (see \cite{lauriere2021}). Alternatively, the Newton  method may be selectively employed to tackle the Hamilton-Jacobi equation at each iteration while using a fixed point iterations for the MFG system (as in e.g. \cite{ghattassi2023non,li2021simple}).  Another approach involves employing a nonlinear discretized system, as presented in \cite{ad,acd}, followed by the application of automatic differentiation for Newton's iteration. 
Within this context, a significant challenge involves establishing a priori estimates for the finite difference scheme, ensuring the stability of the region of attraction of the method  with respect to the discretization parameters. Our convergence result can   be interpreted as an intermediate step in the proof of the convergence  of the Newton  method for   finite dimensional approximation to the MFG system.  However, it is important to note that the convergence analysis presented in this work does not readily extend to the discretized system, and addressing this is left for our future works. 
\par
The paper is organized as follows. In   \cref{sec:prelim}, we introduce some notations and    preliminary results. In    \cref{sec:newton_local}, we discuss the convergence of the Newton method for a non-separable Hamiltonian  and  local coupling, while  in   \cref{sec:newton_nonlocal} the case of a separable Hamiltonian and nonlocal coupling. Finally, the   \cref{sec:appendix} is devoted to the proof of some basic results necessary for the rest of the paper.
\par

\section{Preliminaries}\label{sec:prelim}
In this section, we introduce the assumptions on the Hamiltonian, prove the well posedness of \eqref{MFG} and some preliminary results necessary for the estimate of the rate in the next section.\par
Throughout the paper, we work in the $d-$dimensional torus $\T^d$ (i.e. periodic boundary conditions). We consider the set $\mes(\T^d)$ of Borel probability measures on $\T^d$ is endowed with the Monge Kantorovich (Wasserstein) distance: for $m,m' \in \mes(\T^d)$, ${\bf{d}}_1(m,m')=\sup_\phi \int_{\T^d}\phi (x)d(m-m')(x)$ where the supremum is taken over all Lipschitz maps $\phi: \T^d\rightarrow \R$ with a Lipschitz constant bounded by $1$. In particular, we have that ${\bf{d}}_1(m,m')\leq  \|m-m'\|_{\mathcal C^0(\T^d)}$ if $m,m'\in \mes(\T^d)\cap \mathcal C^0(\T^d)$. Given a map $f:\T^d\times \mes(\T^d)\rightarrow \R^d$, we will use the notation $\frac{\delta f}{\delta m}$ for the derivative of $f$ w.r.t $m$, as introduced in \cite[Section 2.2]{cardaliaguet2019}. $\frac{\delta f}{\delta m}:\states\times \mes(\states)\times \T^d\to \R$ is a continuous map such that
$$
f[m'](x)-f[m](x)= \int_0^1\mint \frac{\delta f}{\delta m}[(1-s)m+sm'](x)(y)d (m'-m)(y) ds.
$$  
The above relation defines the map $\frac{\delta f}{\delta m}$ only up to a constant. We always use the normalization 
$$\int_{\T^d} \frac{\delta f}{\delta m}[m](x)(y)dm(y)=0.$$
 Higher order derivatives are defined similarly.\par
 \par
The set $\mathcal C^{1,0}(Q)$ with the norm 
$$
\|u\|_{\mathcal C^{1,0}(Q)}= \|u\|_{\mathcal C^{0}(Q)}+ \|D u\|_{\mathcal C^{0}(Q;\R^d)}
$$
is the space of continuous functions on $Q$ with continuous derivatives in the $x-$variable, up to the parabolic boundary. 
We also recall the definition of parabolic H\"older spaces on the torus (we refer to \cite{LSU} for a more comprehensive discussion). For $0<\alpha<1$, we denote 
\begin{equation}\label{def Holder}
[u]_{C^{\alpha,\frac{\alpha}{2}}(Q)}:=\sup_{(x_1,t_1),(x_2,t_2)\in Q}\frac{\vert u(x_1,t_1)-u(x_2,t_2)\vert }{(d(x_1,x_2)^2+\vert t_1-t_2\vert )^{\frac{\alpha}{2}}} ,
\end{equation}
where $d(x,y)$ stands for the geodesic distance from $x$ to $y$ in ${\mathbb{T}^d}$. The parabolic H\"older space $C^{\alpha,\frac{\alpha}{2}}(Q)$ is the space of functions $u\in L^\infty(Q)$ for which $[u]_{C^{\alpha,\frac{\alpha}{2}}(Q)}<\infty$. It is endowed with the norm:
\begin{equation*}
\|u\|_{\mathcal C^{\alpha,\frac{\alpha}{2}}(Q)}:=\|u\|_{\mathcal C^{0}(Q)}+[u]_{C^{\alpha,\frac{\alpha}{2}}(Q)}.
\end{equation*}
The space $\mathcal C^{1+\alpha,\frac{1+\alpha}{2}}(Q)$ and $\mathcal C^{2+\alpha,1+\alpha/2}(Q)$ are endowed with the norms
\begin{equation}\label{Holder alpha+1}
\|u\|_{\mathcal C^{1+\alpha,\frac{1+\alpha}{2}}(Q)}:=\|u\|_{\mathcal C^{0}(Q)}+\sum_{i=1}^d\| \partial_{x_i}u\|_{C^{\alpha,\frac{\alpha}{2}}(Q)}+\sup_{(x_1,t_1),(x_2,t_2)\in Q}\frac{\vert u(x_1,t_1)-u(x_2,t_2)\vert }{\vert t_1-t_2\vert ^{\frac{1+\alpha}{2}}} ,
\end{equation}
 \begin{equation}\label{C 2+a}
\|u\|_{\mathcal C^{2+\alpha,1+\alpha/2}(Q)}:=\|u\|_{\mathcal C^0(Q)}+\sum_{i=1}^d\|\frac{\partial u}{\partial x_i}\|_{\mathcal C^{1+\alpha,\frac{1+\alpha}{2}}(Q)}+\|\frac{\partial u}{\partial t}\|_{\mathcal C^{\alpha,\alpha/2}(Q)}.
\end{equation}

We now introduce some useful anisotropic Sobolev spaces to handle time-dependent problems. First, given a Banach space $X$, $L^p(0,T;X)$ denotes the usual vector-valued Lebesgue space for $p\in [1,\infty]$. For any $r\geq1$, we denote by $W^{2,1}_r(Q)$ the space of functions $u$ such that $\partial_t^{{\delta}}D^{\sigma}_x u\in L^r(Q)$ for all multi-indices $\sigma$ and ${\delta}$ such  that $\vert \sigma \vert+2{\delta}\leq  2$, endowed with the norm
\begin{equation*}
	\|u\|_{W^{2,1}_r(Q)}=\Big(\int_{Q}\sum_{\vert \sigma \vert+2{\delta}\leq2}\vert \partial_t^{{\delta}}D^{\sigma}_x u\vert ^rdxdt\Big)^{\frac1r}.
\end{equation*}
We recall that, by classical results in interpolation theory, the sharp  space of initial (or terminal) trace of $W^{2,1}_r(Q)$ is given by the fractional Sobolev class $W^{2-\frac{2}{r}}_r({\mathbb{T}^d})$.\par
We define $W^{1,0}_r(Q)$ as the space of functions on~$Q$ such that the norm
\[
\norm{u}_{W^{1,0}_r(Q)}:=\|u\|_{L^r(Q)}+ \sum_{i=1}^d\|\frac{\partial u}{\partial x_i}\|_{L^r(Q)}
\]
is finite and  we denote with $\mathcal{H}_r^{1}(Q)$ the space of   functions $u\in W^{1,0}_r(Q)$ with $\partial_t u\in (W^{1,0}_{r'}(Q))'$, equipped with the natural norm
\begin{equation*}
	\|u\|_{\mathcal{H}_r^{1}(Q)}:=\|u\|_{W^{1,0}_r(Q)}
	+\norm{\partial_tu}_{(W^{1,0}_{s'}(Q))'}\ .
\end{equation*}
From \cite[Theorem A.3 (iii)]{Meta} and \cite[Proposition 2.1 (iii)]{Cirant2019}, for $r>d+2$ the space $\mathcal H^1_r(Q)$ is continuously embedded in $\mathcal C^{\alpha/2,\alpha}(Q)$, for some $\alpha\in (0,1)$. \par

We consider the following set of assumptions for the non-separable case with   local coupling, while specific assumptions in the case of a nonlocal coupling will be discussed in Section \ref{sec:newton_nonlocal}. The notation $|\cdot|$ both refers to the modulus of a vector and to the norm of a matrix in the appropriate space.
\begin{itemize}
\item[(A1)]  $m_0\in \mes(\T^d) \cap \mathcal C^{2+\alpha}(\T^d)$ and $m_0(x)\geq \vartheta >0$, 
  $u_T\in \mathcal C^{2+\alpha}(\states)$.
	\item[(A2)] $H\in \mathcal C^4(\T^d\times \R^+ \times \R^d)$ and for all $x\in \T^d$, $m\in \R^+$, $p\in \R^d$ and some $\bar C>0$: 
\begin{equation}\label{H}
\begin{gathered}
|H_{px}(x,m,p)|\leq \bar C(|p| + 1), \,\,\,|H_{xx}(x,m,p)|+|H_{xxm}(x,m,p)| \leq \bar C (|p|^2 + 1),\\
 |H_{mp}(x,m,p)|\leq \bar C\vert p\vert,\,\,|H_{mm}(x,m,p)|\leq \bar C\vert p\vert^2,\\  |H_{pp}(x,m,p) | + |H_{ppm}(x,m,p) |\leq \bar C,\\
 |H_{ppp}(x,m,p)|+|H_{pppm}(x,m,p)|  \leq \bar C .
\end{gathered}
\end{equation}	
\begin{equation}\label{H Hessian}
 \begin{pmatrix}
-H_m(x,m,p) & \dfrac{m}{2}H_{pm}(x,m,p)^T\\[5pt]
\dfrac{m}{2}H_{pm}(x,m,p) & mH_{pp}(x,m,p)
\end{pmatrix}
>0,\quad\forall m>0.
\end{equation}
\item[(A3)]  $f(\cdot)$, $f'(\cdot)$ and $f''(\cdot)$ are uniformly bounded mappings from $\R^+$ to $\R$. Moreover, $f'(\cdot)\geq 0$.
\end{itemize}
 Some remarks about these assumptions are in order.
 \begin{remark}
\cref{H Hessian},   first proposed by P. L. Lions in \cite{Lions} and then exploited in \cite{achdou2018,achdou2015}, 
is a uniqueness condition for  the MFG systems with   non-separable Hamiltonian. In particular, it implies   that $H$ is convex with respect to $p$ and nonincreasing with respect to $m$ and, when $H$ has a separate form  $H = ~H( x, p)-\overline f(m)$, it reduces to $H_{pp}> 0$ and $\overline f '>0$. Besides for uniqueness, we use this condition to prove the  estimate in \cref{v rho estimate}, which is crucial for the rate of convergence.
 \end{remark}
\begin{remark}
A typical example of Hamiltonian which satisfies (A2) is 
\begin{equation}\label{quadratic H}
H(x,m,p)=\frac{h(x)|p|^2}{(1+m)^\alpha},
\end{equation}
where   $0<\alpha\leq 2$, $h(x)\in \mathcal C^2(\T^d)$ and $h(x)>0$ for all $x\in \T^d$. Existence and uniqueness of a weak solutions to MFGs with such Hamiltonians, under some additional assumptions, can be obtained from results in \cite{achdou2018, graber2015weak}. In \cref{Well posed}, under  the stronger assumptions (A1)-(A3), we prove existence and uniqueness of a classical solution to \cref{MFG}. 
\end{remark}
\begin{remark}
    An example of $f$ which satisfies (A3) is the sigmoid function
	$$
	f(m)=\frac{1}{1+e^{-m}}.
	$$
	\\
	In fact, the uniformly boundedness of $f$ is included in (A3) only to obtain a relatively  simple proof of existence  of a solution in the non-separable case, see   \cref{Well posed}, but  one can obtain small time existence and uniqueness  results  with less restrictions on $H$ and $f$, see \cite{cirant2020}. If we assume a priori that \cref{MFG} has a classical solution, independently of assuming  (A1)-(A3),  the key assumption
	for proving the convergence of Newton method     is the uniform boundedness of $f''(\cdot)$, see also \cref{rem_classical}. In this case, we can also include examples such as
	\begin{itemize}
		\item $f(m)=m$.
		\item $f(m)=(1+m)^\alpha$, $0<\alpha<2$.
	\end{itemize}
	Therefore, in some practical applications, we can apply the Newton method without requiring all the assumptions in this paper to be satisfied.  In any case,  restrictions on the uniform boundedness of $H_{ppp}(x,m,p)$ and $f''(m)$ are very typical for Newton iterations, c.f. S. Boyd \cite[Section 9.5.3]{boyd2004}. Some possible generalizations to the coupling $f(m)=m^\alpha$, $\alpha\geq 2$, will be discussed later in the paper, see \cref{rem_f}. It is also possible to include, under appropriate assumptions, a dependence of $f$ on $t$, but for simplicity we omit it.
\end{remark}

We will consider classical solutions to the MFG system. Recall that a classical solution of \eqref{MFG} is a couple $(u,m)$ such that $u$ and $m$ belong to $\mathcal C^{2, \alpha}(Q)$ for some $\alpha\in (0,1)$ and satisfies the problem in pointwise sense.
For the proof of the next result, see the Appendix.  
\begin{proposition}\label{Well posed}
	Under assumptions (A1), (A2) and (A3), the system \eqref{MFG} has a unique classical solution. \\
	
\end{proposition}
The next two lemmas are devoted to prove an estimate for a perturbation of the linearized MFG system. This result is   the main ingredient in our analysis of the convergence of the Newton algorithm.  
\begin{lemma}\label{stable}
Assume (A1), (A2) and (A3) and let $(u,m)$ be the the solution of \cref{MFG}. Then,  the unique weak solution $(v,\rho)$ of the  system 
	\be\label{linearized}
	\left\{\begin{aligned}
		(i) \qquad &-\partial_t v -\Delta v+H_p(x,m,Du)Dv+H_m(x,m,Du)\rho=f'(m)\rho \qquad &&{\rm in}\,\,Q,\\
		(ii) \qquad &\partial_t \rho- \Delta \rho  - {\rm{div}} \big( \rho H_p(x,m,Du)\big)-{\rm{div}} \big( m\rho H_{pm}(x,m,Du)\big) ={\rm{div}}\big(mH_{pp}(x,m,Du)Dv\big) &&{\rm in}\,\,Q, \\
		& \rho(x,0)=0 , \; v(x,T)=0 &&{\rm in}\,\,\T^d
	\end{aligned}\right.
	\ee
	is the trivial solution $(v,\rho)=(0,0)$. 
\end{lemma}
\begin{proof}
	Multiply by $\rho$ on both sides of (i), integrate on $Q$ and exploit (ii) to get 
	\begin{equation*}
	\begin{aligned}
	&\int_Qf'(m)|\rho|^2dxdt\\
	={}&\int_QH_m(x,m,Du)|\rho|^2dxdt+\int_Qv{\rm{div}}\big(mH_{pp}(x,m,Du)Dv\big)dxdt+\int_Qv{\rm{div}}\big(m\rho H_{pm}(x,m,Du)\rho \big)dxdt\\
	={}&\int_QH_m(x,m,Du)|\rho|^2dxdt-\int_QmH_{pp}(x,m,Du)Dv\cdot Dvdxdt-\int_QmH_{pm}(x,m,Du)\rho Dv.
	\end{aligned}
	\end{equation*}
	It follows from (A1) and parabolic maximum principle (c.f. \cite{tang2022learning}) that $m>0$. Hence with (A3) we obtain $f'(m)\geq 0$ and 
	$$
 \int_Qf'(m)|\rho|^2dxdt\geq 0.
	$$
Hence, from \cref{H Hessian}, we get that
	\begin{equation*}
\begin{pmatrix}
\rho & Dv
\end{pmatrix}
 \begin{pmatrix}
-H_m(x,m,Du) & mH_{pm}(x,m,Du)/2\\[4pt]
mH_{pm}(x,m,Du)/2 & mH_{pp}(x,m,Du)
\end{pmatrix}
\begin{pmatrix}
\rho \\
Dv
\end{pmatrix}
= 0,
\end{equation*}
otherwise we obtain a contradiction. Therefore $(\rho,Dv)\equiv(0,0)$. From $v(x,T)=0$, it also follows  that $v=0$.
\end{proof}
The estimate in the next lemma is similar to \cite[Lemma 5.2]{bc}, with the key differences that we consider non-separable Hamiltonian and local couplings.

\begin{lemma}\label{v rho estimate}
	Assume (A1), (A2) and (A3) and let $(u,m)$ be the classical solution of \cref{MFG}.
	Given $a\in \mathcal{C}^0(Q)$ and a vector field  $b\in \mathcal C^0(Q;\R^d)$, let $(v,\rho)$ be a classical solution of the perturbed linear system
	\be\label{v rho}
	\left\{\begin{aligned}
		(i)\qquad &-\partial_t v -\Delta v+H_p(x,m,Du)Dv+H_m(x,m,Du)\rho=f'(m)\rho+a(x,t) \qquad &&{\rm in}\,\,Q,\\
		(ii)\qquad &\partial_t \rho- \Delta \rho  - {\rm{div}} \big( \rho H_p(x,m,Du)\big) -{\rm{div}} \big( m\rho H_{pm}(x,m,Du)\big) \\
		={}& {\rm{div}}\big(mH_{pp}(x,m,Du)Dv\big)+{\rm{div}}(b(x,t)) &&{\rm in}\,\,Q, \\
		& \rho(x,0)=0 , \; v(x,T)=0  &&{\rm in}\,\,\T^d.
	\end{aligned}\right.
	\ee
	Then, there exists  a constant $C>0$,  depending only on the coefficients of the problem, such that
	\begin{equation}\label{est}
		\|v\|_{\mathcal{C}^{1,0}}+\|\rho\|_{\mathcal{C}^{0}}\leq C\Big(\|a\|_{\mathcal{C}^{0}}+\|b\|_{\mathcal{C}^{0}}\Big).
	\end{equation}
\end{lemma}
\begin{proof}
	First observe that since the system \eqref{v rho} is linear,
	then $(v,\rho)/(\|a\|_{\mathcal{C}^{0}}+\|b\|_{\mathcal{C}^{0}})$ is the solution of the problem corresponding to the perturbation $(a,b)/(\|a\|_{\mathcal{C}^{0}}+\|b\|_{\mathcal{C}^{0}})$. Hence \eqref{est} is equivalent to show  that,  for $\|a\|_{\mathcal{C}^{0}}+\|b\|_{\mathcal{C}^{0}}\leq 1$,  then $\|v\|_{\mathcal{C}^{1,0}}+\|\rho\|_{\mathcal{C}^{0}}\leq C$ for some $C>0$.\par
	We argue by contradiction and suppose that the estimate is not true. Hence  we assume that there exists $a^k$, $b^k$ and $(v^k,\rho^k)$ with 
	\begin{equation}\label{contradiction}
		\|a^k\|_{\mathcal{C}^{0}}+\|b^k\|_{\mathcal{C}^{0}}\leq 1,\,\,\theta^k:=\|v^k\|_{\mathcal{C}^{1,0}}+\|\rho^k\|_{\mathcal{C}^{0}}\geq k.
	\end{equation}
	Set 
	\begin{equation*}
		\tilde{v}^k:=\frac{v^k}{\theta^k},\,\,\,\tilde{\rho}^k:=\frac{\rho^k}{\theta^k}.
	\end{equation*}
	By definition, $\|\tilde{v}^k\|_{\mathcal{C}^{1,0}}+\|\tilde{\rho}^k\|_{\mathcal{C}^{0}}=1$ for all $k$ and the pair $(\tilde{v}^k,\tilde{\rho}^k)$ solves:
	\be
	\left\{\begin{aligned}
		(i) \qquad &-\partial_t \tilde{v}^k -\Delta \tilde{v}^k+H_p(x,m,Du)D\tilde{v}^k+H_m(x,m,Du)\tilde{\rho}^k=f'(m)\tilde{\rho}^k+\frac{a^k(x,t)}{\theta^k} \qquad &&{\rm in}\,\,Q,\\
		(ii) \qquad &\partial_t \tilde{\rho}^k- \Delta \tilde{\rho}^k  - {\rm{div}} \big( \tilde{\rho}^k H_p(x,m,Du)\big) -{\rm{div}} \big( m\tilde{\rho}^k H_{pm}(x,m,Du)\big) \\
		={}&{\rm{div}}\big(mH_{pp}(x,m,Du)D\tilde{v}^k\big)+{\rm{div}}\big(\frac{b^k(x,t)}{\theta^k}\big) &&{\rm in}\,\,Q, \\
		& \rho(x,0)=0 , \; v(x,T)=0 &&{\rm in}\,\,\T^d.
	\end{aligned}\right.
	\ee
	Observe that $\tilde{v}^k$ is a solution of a linear parabolic equation with bounded coefficients. Hence,   $\tilde{v}^k$ and  $D\tilde{v}^k$ are bounded in $\mathcal C^{\alpha,\alpha/2}$ for some $\alpha \in (0,1)$. Similarly, $\tilde{\rho}^k$, solution of a linear equation in divergence form,  is bounded in $\mathcal C^{\beta,\beta/2}$ for some $\beta \in (0,1)$. 
	By taking subsequences we obtain a cluster point $(v,\rho)$ of $(\tilde{v}^k,\tilde{\rho}^k)$ such that 
	\begin{equation}\label{norm}	 
	\|v\|_{\mathcal{C}^{1,0}}+\|\rho\|_{\mathcal{C}^{0}}=\lim_{k\rightarrow +\infty}(\|\tilde{v}^k\|_{\mathcal{C}^{1,0}}+\|\tilde{\rho}^k\|_{\mathcal{C}^{0}})=1.	 
	\end{equation}
	  By   \cref{contradiction}, we know $ a^k(x,t)/\theta^k$ and ${ \rm{div}} ( b^k(x,t)/\theta^k )$ actually vanish for $k\to \infty$ and therefore  $(v,\rho)$ is a solution of  \eqref{linearized}. Hence by \cref{stable} we have $(v,\rho)=(0,0)$, a contradiction to \cref{norm}. 
\end{proof}

\section{The Newton  method for the Mean Field Games system with non-separable Hamiltonian  and local coupling}\label{sec:newton_local}
In this section, we give an estimate for the rate of convergence of the Newton method to the MFG system in the case of a non-separable Hamiltonian and local coupling. We first prove the well posedness of the system \eqref{Newton system} for each $n$.
\begin{proposition}\label{system n}
	For any $n\in\N$, there exists  a unique  solution $(u^{n},m^{n})\in \mathcal C^{2+\alpha,1+\alpha/2}(Q)\times \mathcal C^{2+\alpha,1+\alpha/2}(Q)$  to the system \eqref{Newton system}.
\end{proposition}
\begin{proof}
	Assume to have proved the statement at step $n-1$. Hence, given $(u^{n-1},m^{n-1})\in \mathcal C^{2+\alpha,1+\alpha/2}(Q)\times \mathcal C^{2+\alpha,1+\alpha/2}(Q)$,  \cref{Newton system} is a strongly coupled linear   system for $(u^{n},m^{n})$.\\
	We first prove existence of a weak solution $(u^n,m^n)\in W^{2,1}_r(Q)\times \mathcal{H}^1_r(Q)$, $r>d+2$ by means of  a fixed point argument. \par
	Define $\mathbf{X}:=\{\varrho \in \mathcal C^0(Q): \varrho\ge 0, \varrho(x,0)=m_0(x), \int_{\T^d}\varrho(x,t)dx=1\}$ and consider the compact mapping $\hat{\varrho}=\mathbf{T}(\varrho): \mathcal C^0(Q)\rightarrow \mathcal{C}^{\alpha,\alpha/2}(Q)$ defined by solving in sequence
\be\label{new1}
\left\{\begin{aligned}
	(i) \qquad &-\partial_t \hat{u} -\Delta \hat{u}+H_p(x,m^{n-1},Du^{n-1})D\hat{u}+H_m(x,m^{n-1},Du^{n-1}) (\varrho- m^{n-1})\\
	={}&H_p(x,m^{n-1},Du^{n-1})Du^{n-1}-H(x,m^{n-1},Du^{n-1})+f(m^{n-1})+f'(m^{n-1})(\varrho-m^{n-1}) \qquad &&{\rm in}\,\,Q,\\
  & \hat{u}(x,T)=u_T(x) &&{\rm in}\,\,\T^d,\\
	(ii) \qquad &\partial_t \hat{\varrho}- \Delta \hat{\varrho}  - \text{div} \big( \varrho H_p(x,Du^{n-1})\big)-\text{div} \big(m^{n-1} H_{pm}(x,m^{n-1},Du^{n-1})\varrho\big)\\
	={}& {\rm{div}}\big(m^{n-1}H_{pp}(x,m^{n-1},Du^{n-1})(D\hat{u}-Du^{n-1})\big)-\text{div} \big((m^{n-1})^2H_{pm}(x,m^{n-1},Du^{n-1})\big) &&{\rm in}\,\,Q, \\
	&\hat{\varrho}(x,0)=m_0(x) &&{\rm in}\,\,\T^d.
\end{aligned}\right.
\ee
We rewrite equation (i) in \cref{new1} as
$$
-\partial_t \hat{u} -\Delta \hat{u}+H_p(x,m^{n-1},Du^{n-1})D\hat{u}=\mathsf{f}
$$
with
 \begin{equation}\label{sf f}
	\begin{aligned}
	\mathsf{f}:={}&H_p(x,m^{n-1},Du^{n-1})Du^{n-1}-H(x,m^{n-1},Du^{n-1})+f(m^{n-1})\\
	 &+f'(m^{n-1})(\varrho-m^{n-1})-H_m(x,m^{n-1},Du^{n-1}) (\varrho- m^{n-1}) \in L^\infty(Q).
	\end{aligned}
	\end{equation}
By \cref{linear estim Sobolev},	we obtain    the existence of $\hat{u}\in W^{2,1}_r(Q)$ solving (i) in \eqref{new1}. By Sobolev embedding $D\hat{u} \in \mathcal{C}^{\alpha,\alpha/2}(Q;\R^d)$. Next we rewrite  equation (ii) in \cref{new1} as
	$$
	\partial_t \hat{\varrho}- \Delta \hat{\varrho}  - {\rm div}(\mathsf{F})=0
	$$
with 
\begin{align*}
	\mathsf{F}:={}&\varrho H_p(x,Du^{n-1})+m^{n-1} H_{pm}(x,m^{n-1},Du^{n-1})+m^{n-1}H_{pp}(x,m^{n-1},Du^{n-1})(D\hat{u}-Du^{n-1})\\
	&-(m^{n-1})^2H_{pm}(x,m^{n-1},Du^{n-1}).
\end{align*}
From $D\hat{u} \in \mathcal{C}^{\alpha,\alpha/2}(Q;\R^d)$ and the assumptions on $(u^{n-1},m^{n-1})$, $\mathsf{F}\in L^\infty(Q;\R^d)$ and, by \cref{m stability}, we obtain  there exists a solution $\hat{\varrho}\in \mathcal{H}^1_r(Q)$ to (i) in \eqref{new1}, thus $\hat{\varrho}\in \mathcal{C}^{\alpha,\alpha/2}(Q)$. \\		
	Set $\delta \hat{\varrho}=\hat{\varrho_1}-\hat{\varrho_2}$, $\delta \hat{u}=\hat{u}_1-\hat{u}_2$. Then
	\be
	\left\{\begin{aligned}
		(i)  \qquad &-\partial_t \delta \hat{u} -\Delta \delta \hat{u}+H_p(x,m^{n-1},Du^{n-1})D\delta \hat{u}+m^{n-1} H_{pm}(x,m^{n-1},Du^{n-1})\delta \varrho \\
		={}&f'(m^{n-1})\delta  \varrho \qquad &&{\rm in}\,\,Q,\\
		(ii) \qquad &\partial_t \delta \hat{\varrho}- \Delta \delta \hat{\varrho}  - \text{div} \big( \delta \varrho H_p(x,m^{n-1},Du^{n-1})\big)-\text{div} \big(m^{n-1} H_{pm}(x,m^{n-1},Du^{n-1})\delta \varrho \big)\\
		={}&{\rm{div}}\big(m^{n-1}H_{pp}(x,m^{n-1},Du^{n-1})D\delta \hat{u}\big) &&{\rm in}\,\,Q, \\
		&\delta \hat{\varrho}(x,0)=0, \; \delta \hat{u}(x,T)=0 &&{\rm in}\,\,\T^d. 
	\end{aligned}\right.
	\ee
	We obtain by \cref{linear estim Sobolev} that
	$$
	\|\delta \hat{u}\|_{W^{2,1}_r(Q)}\leq C\|\delta  \varrho\|_{L^\infty(Q)},\,\,\,\|D\delta \hat{u}\|_{L^\infty(Q;\R^d)}\leq C\|\delta  \varrho\|_{L^\infty(Q)},
	$$
	then by \cref{m stability}
	$$
	\|\delta \hat{\varrho}\|_{\mathcal{C}^{\alpha,\alpha/2}}\leq C(\|\delta  \varrho\|_{L^\infty(Q)}+\|D\delta \hat{u}\|_{L^\infty(Q;\R^d)})\leq C\|\delta  \varrho\|_{L^\infty(Q)}.
	$$
	It then follows that $\mathbf{T}$ is a continuous map. We conclude, by Schauder fixed point theorem, the existence of a solution to \eqref{Newton system}. 
	It follows that $(\hat{u},\hat{\varrho})\in W^{2,1}_r(Q)\times \mathcal{C}^{\alpha,\alpha/2}(Q)$ is a fixed point defined by \cref{new1},  with $\mathsf{f}$ replaced by   
\begin{equation*}
 \begin{aligned}
\hat{\mathsf{f}}:={}&H_p(x,m^{n-1},Du^{n-1})Du^{n-1}-H(x,m^{n-1},Du^{n-1})+f(m^{n-1})\\
	 &+f'(m^{n-1})(\hat{\varrho}-m^{n-1})-H_m(x,m^{n-1},Du^{n-1}) (\hat{\varrho}- m^{n-1}).
 \end{aligned}
\end{equation*}
 Since $\hat{\mathsf{f}}\in \mathcal{C}^{\alpha,\alpha/2}(Q)$, from \cref{linear estim} it follows that $\hat{u}\in \mathcal C^{2+\alpha,1+\alpha/2}(Q)$.
	By assumption (A2), $\hat{u}\in \mathcal C^{2+\alpha,1+\alpha/2}(Q)$ and $(u^{n-1},m^{n-1})\in \mathcal C^{2+\alpha,1+\alpha/2}(Q)\times \mathcal C^{2+\alpha,1+\alpha/2}(Q)$, we obtain ${\rm div}(\mathsf{F})\in \mathcal{C}^{\alpha,\alpha/2}(Q)$. Using \cref{linear estim} again we obtain 
	$\hat{\rho}\in \mathcal C^{2+\alpha,1+\alpha/2}(Q)$.\par
  	We now prove uniqueness. Assume that there are two solutions  $(\hat u_1,\hat\rho_1)$ and   $(\hat u_2,\hat\rho_2)$ to \cref{new1} and set $\delta \hat{u}=\hat{u}_1-\hat{u}_2$, $\delta \hat{\varrho}=\hat{\varrho_1}-\hat{\varrho_2}$.  Clearly $(\delta \hat u,\delta \hat \rho)$ solves 
	\be\label{linear}
	\left\{\begin{aligned}
		(i)  \qquad &-\partial_t \delta \hat{u} -\Delta \delta \hat{u}+H_p(x,m^{n-1},Du^{n-1})D\delta \hat{u}+m^{n-1} H_{m}(x,m^{n-1},Du^{n-1})\delta \hat \rho \\
		={}&f'(m^{n-1})\delta \hat \rho \qquad&&{\rm in}\,\,Q,\\
		(ii) \qquad &\partial_t \delta \hat{\varrho}- \Delta \delta \hat{\varrho}  - \text{div} \big( \delta \hat \rho H_p(x,m^{n-1},Du^{n-1})\big)-\text{div} \big(m^{n-1} H_{pm}(x,m^{n-1},Du^{n-1})\delta \hat \rho \big)\\
		={}&{\rm{div}}\big(m^{n-1}H_{pp}(x,m^{n-1},Du^{n-1})D\delta \hat{u}\big) &&{\rm in}\,\,Q, \\
		&\delta \hat{\varrho}(x,0)=0, \; \delta \hat{u}(x,T)=0 &&{\rm in}\,\,\T^d.
	\end{aligned}\right.
	\ee
	Testing the equation (i) with $\hat \rho$, equation (ii)  with $\hat u$ and subtracting the resulting identities,   an easy computation  gives
	\begin{equation*}
	\begin{aligned}
	&\int_Qf'(m^{n-1})|\delta \hat{\varrho}|^2dxdt\\
	={}&\int_QH_m(x,m^{n-1},Du^{n-1})|\delta \hat{\varrho}|^2dxdt-\int_Qm^{n-1}H_{pp}(x,m^{n-1},Du^{n-1})D\delta \hat{u}\cdot D\delta \hat{u}dxdt\\
	&-\int_Qm^{n-1}H_{pm}(x,m^{n-1},Du^{n-1})\delta \hat{\varrho} D\delta \hat{u}\leq 0.
	\end{aligned}
	\end{equation*}
	By (A2), \cref{H Hessian}, and (A3) we get $(\delta \hat{\varrho},D\delta \hat{u})=(0,0)$. From $\delta \hat{u}(x,T)=0$, it also follows  that $\delta \hat{u}=0$.	
\end{proof} 
\cref{system n} is concerned with the invertibility of the linear operator $\mathcal{LF}$ at each step $n$, as defined in \cref{LF}. Hence, it is not surprising that one needs some Hessian type condition. The constant $C$ in \cref{system n} may depend  on the previous step $(u^{n-1},m^{n-1})$. In the discretized setting, invertibility of a    linearized system similar to \cref{linear} has been discussed in \cite[Section 4.1]{acdplanning} for solving a mean field planning problem with separable Hamiltonian. We believe the ideas in \cite[Section 4.1]{acdplanning} can be extended also to MFGs with non-separable Hamiltonian. However, solvability at each iteration is not enough for the convergence of the  Newton  method since it is necessary to prove some a priori estimates independent of iteration index $n$.\par
In the next result, we  obtain the local quadratic rate of convergence result.
\begin{theorem}\label{Main Thm local}
Let $(u,m)$ be  the solution of system \eqref{MFG} and $(u^n,m^n)$ is the sequence generated by Newton's algorithm \eqref{Newton system}. Set $v^n=u^n-u$, $\rho^n=m^n-m$. There exists a constant $\eta>0$ such that if $\|v^0\|_{\mathcal{C}^{1,0}}+\|\rho^0\|_{\mathcal{C}^{0}}\leq \eta$ then $\|v^n\|_{\mathcal{C}^{1,0}}+\|\rho^n\|_{\mathcal{C}^{0}}\rightarrow 0$ with  a quadratic rate of convergence.
\end{theorem}
\begin{proof}  
We emphasize that from here and for the rest of the proof, $C$ denotes some generic constant which may increase from line to line.  This constant may depend on data of the problem and the solution $(u,m)$, but it is always independent of $n$.\par
We observe that $v^n$ solves the equation
\begin{equation*}
\begin{aligned}
&-\partial_t v^n -\Delta v^n+H_p(x,m,Du)\cdot Dv^n\\
={}&f(m^{n-1})-f(m)+f'(m^{n-1})(m^n-m^{n-1})+H(x,m,Du)-H(x,m^{n-1},Du^{n-1})\\
&-H_p(x,m^{n-1},Du^{n-1})D(u^{n}- u^{n-1})-H_m(x,m^{n-1},Du^{n-1}) (m^{n}- m^{n-1}),
\end{aligned}
\end{equation*}
which can be rewritten as
$$
-\partial_t v^n -\Delta v^n+H_p(x,m,Du)\cdot Dv^n+H_m(x,m,Du)\rho^n=f'(m)\rho^n+a,
$$
where
\begin{equation}\label{stima_a}
\begin{aligned}
a:={}& H_p(x,m,Du)(Du^n-Du)+H_m(x,m,Du)(m^n-m)+H(x,m,Du)-H(x,m^{n-1},Du^{n-1})\\
&-H_p(x,m^{n-1},Du^{n-1})D(u^{n}- u^{n-1})-H_m(x,m^{n-1},Du^{n-1}) (m^{n}- m^{n-1})\\
&-f'(m)(m^n-m)+f(m^{n-1})-f(m)+f'(m^{n-1})(m^n-m^{n-1}).
\end{aligned}
\end{equation}
In order to apply     \cref{v rho estimate}, we need to estimate $\|a\|_{\mathcal{C}^0}$. We rewrite the terms involving $H$ in \cref{stima_a} as
\begin{equation}\label{stima0}
\begin{aligned}
&H_p(x,m,Du)(Du^n-Du)+H_m(x,m,Du)(m^n-m)+H(x,m,Du)-H(x,m^{n-1},Du^{n-1})\\
&-H_p(x,m^{n-1},Du^{n-1})D(u^{n}- u^{n-1})-H_m(x,m^{n-1},Du^{n-1}) (m^{n}- m^{n-1})\\
={}&H(x,m,Du)-H(x,m^{n-1},Du^{n-1})-H_p(x,m^{n-1},Du^{n-1})D(u- u^{n-1})\\
&-H_m(x,m^{n-1},Du^{n-1}) (m- m^{n-1})+\Big(H_p(x,m,Du)-H_p(x,m^{n-1},Du^{n-1})\Big)(Du^n-Du)\\
&+\Big(H_m(x,m,Du)-H_m(x,m^{n-1},Du^{n-1})\Big) (m^n-m).
\end{aligned}
\end{equation}
We now estimate the terms on the right hand side of the previous identity. It is clear from (A2) that for any $\tau \in (0,1)$,
\begin{equation}\label{Hmm}
\begin{aligned}
 \sup_\tau& \big\vert H_{mm}(x,m+\tau(m^{n-1}-m),Du+\tau(Du^{n-1}-Du))\big\vert\\
& \le C|Du+\tau(Du^{n-1}-Du)|^2 \leq C\big(\vert Du\vert+\vert Du^{n-1}-Du\vert \big)^2\\
& \le C\big(2\vert Du\vert^2+2\vert Du^{n-1}-Du\vert^2 \big).
\end{aligned}
\end{equation}
Likewise, 
\begin{equation}\label{Hmp}
\begin{aligned}
\sup_\tau \big\vert H_{mp}(x,m+\tau(m^{n-1}-m),Du+\tau(Du^{n-1}-Du))\big \vert  &\leq C|Du+\tau(Du^{n-1}-Du)|\\
&\leq   C\big(\vert Du\vert+\vert Du^{n-1}-Du\vert \big),
\end{aligned}
\end{equation}
and
\begin{equation*}
\sup_\tau \vert H_{pp}(x,m+\tau(m^{n-1}-m),Du+\tau(Du^{n-1}-Du))\vert \leq \bar C.
\end{equation*}
Moreover, by the mean value theorem, we have
\begin{equation}\label{stima1}
\begin{aligned}
&H(x,m,Du)-H(x,m^{n-1},Du^{n-1})-H_p(x,m^{n-1},Du^{n-1})D(u- u^{n-1})\\
&-H_m(x,m^{n-1},Du^{n-1}) (m- m^{n-1})\\
={}&\int_0^1\Big(H_p\big(x,m^{n-1}+\tau(m-m^{n-1}),Du^{n-1}+\tau(Du-Du^{n-1})\big)-H_p(x,m^{n-1},Du^{n-1})\Big)(Du-Du^{n-1})d\tau\\
&+\int_0^1\Big(H_m\big(x,m^{n-1}+\tau(m-m^{n-1}),Du^{n-1}+\tau(Du-Du^{n-1})\big)-H_p(x,m^{n-1},Du^{n-1})\Big)(m-m^{n-1})d\tau.
\end{aligned}
\end{equation}
By using mean value theorem again, we estimate the integrand in \cref{stima1}  by
\begin{equation*}
\begin{aligned}
&H_p\big(x,m^{n-1}+\tau(m-m^{n-1}),Du^{n-1}+\tau(Du-Du^{n-1})\big)-H_p(x,m^{n-1},Du^{n-1})\\
= {}& \int_0^1 \Big(H_{pp}\big(x,m^{n-1}+(1-\tau')\tau(m-m^{n-1}),Du^{n-1}+(1-\tau')\tau(Du-Du^{n-1})\big)\tau(Du^{n-1}-Du)d\tau' \\
&+\int_0^1 \Big(H_{pm}\big(x,m^{n-1}+(1-\tau')\tau(m-m^{n-1}),Du^{n-1}+(1-\tau')\tau(Du-Du^{n-1})\big)\\
&-H_p(x,m^{n-1},Du^{n-1})\Big)\tau(m^{n-1}-m)d\tau 
\end{aligned}
\end{equation*}
and therefore, as $0<(1-\tau')\tau<1$, $0<\tau<1$, we get
\begin{equation*}
\begin{aligned}
&\sup_\tau \vert H_p\big(x,m^{n-1}+\tau(m-m^{n-1}),Du^{n-1}+\tau(Du-Du^{n-1})\big)-H_p(x,m^{n-1},Du^{n-1})\vert\\
\leq {}& \sup_\tau \vert H_{pp}\big(x,m^{n-1}+\tau(m-m^{n-1}),Du^{n-1}+\tau(Du-Du^{n-1})\vert \vert Du-Du^{n-1}\vert \\
&+\sup_\tau \vert H_{pm}\big(x,m^{n-1}+\tau(m-m^{n-1}),Du^{n-1}+\tau(Du-Du^{n-1})\vert \vert m-m^{n-1}\vert\\
\leq {}&\bar C\vert Du-Du^{n-1}\vert +C(\vert Du\vert+\vert Du-Du^{n-1}\vert) \vert m-m^{n-1}\vert,
\end{aligned}
\end{equation*}
\begin{equation*}
\begin{aligned}
&\left\vert \int_0^1\Big(H_p\big(x,m^{n-1}+\tau(m-m^{n-1}),Du^{n-1}+\tau(Du-Du^{n-1})\big)-H_p(x,m^{n-1},Du^{n-1})\Big)(Du-Du^{n-1})d\tau \right\vert\\
\leq {}& (\int_0^1\tau d\tau)\Big(\bar C\vert Du-Du^{n-1}\vert +C(\vert Du\vert+\vert Du-Du^{n-1}\vert) \vert m-m^{n-1}\vert\Big)\vert Du-Du^{n-1}\vert.
\end{aligned}
\end{equation*}
In a similar way, we obtain
\begin{equation*}
\begin{aligned}
&\left\vert \int_0^1\Big(H_m\big(x,m^{n-1}+\tau(m-m^{n-1}),Du^{n-1}+\tau(Du-Du^{n-1})\big)-H_m(x,m^{n-1},Du^{n-1})\Big)(m-m^{n-1})d\tau \right\vert \\
\leq {}& (\int_0^1\tau d\tau)\Big(C(1+\vert Du-Du^{n-1}\vert^2 )\vert m-m^{n-1}\vert+C(1+\vert Du-Du^{n-1}\vert)\vert Du-Du^{n-1}\vert \vert\Big)\vert m-m^{n-1}\vert.
\end{aligned}
\end{equation*}
Therefore, replacing in \cref{stima1}, we have
\begin{equation}\label{stima2}
\begin{aligned}
&\left\vert H(x,m,Du)-H(x,m^{n-1},Du^{n-1})-H_p(x,m^{n-1},Du^{n-1})D(u- u^{n-1})-H_m(x,m^{n-1},Du^{n-1}) (m- m^{n-1})\right\vert\\
&\leq C(\vert Dv^{n-1}\vert^2)+C(1+\vert Dv^{n-1}\vert)\vert Dv^{n-1}\vert \vert \rho^{n-1}\vert+C(1+\vert Dv^{n-1}\vert^2)\vert \rho^{n-1}\vert^2\\
&\leq C\Big(\vert Dv^{n-1}\vert^2+\vert \rho^{n-1}\vert^2+\vert Dv^{n-1}\vert^4+\vert \rho^{n-1}\vert^4\Big).
\end{aligned}
\end{equation}
For the other terms in \cref{stima0}, we first observe that by (A2) and \cref{Hmp} we get
\begin{equation*}
\begin{aligned}
\big \vert& H_p(x,m,Du)-H_p(x,m^{n-1},Du^{n-1})\big \vert \leq {}  \sup_\tau \vert H_{pm}(x,m+\tau(m^{n-1}-m),Du+\tau(Du^{n-1}-Du))  \vert \vert \rho^{n-1}\vert\\
 &+\sup_\tau \vert H_{pp}(x,m+\tau(m^{n-1}-m),Du+\tau(Du^{n-1}-Du))  \vert \vert v^{n-1}\vert\\
 &\leq C(\vert Du\vert+\vert Dv^{n-1}\vert)\vert \rho^{n-1}\vert+C\vert Dv^{n-1}\vert ,
\end{aligned}
\end{equation*}
and,  by \cref{Hmm} and \cref{Hmp},
\begin{equation*}
\begin{aligned}
\big \vert& H_m(x,m,Du)-H_m(x,m^{n-1},Du^{n-1})\big \vert \leq   \sup_\tau \vert H_{pm}\big(x,m+\tau(m^{n-1}-m),Du+\tau(Du^{n-1}-Du)\big)  \vert \vert v^{n-1}\vert\\
 &+\sup_\tau \vert H_{mm}\big(x,m+\tau(m^{n-1}-m),Du+\tau(Du^{n-1}-Du)\big)  \vert \vert \rho^{n-1}\vert\\
 &\leq C(\vert Du\vert+\vert Dv^{n-1}\vert)\vert Dv^{n-1}\vert+C(\vert Du\vert^2+\vert Dv^{n-1}\vert^2)\vert \rho^{n-1}\vert.
\end{aligned}
\end{equation*}
We   then obtain
\begin{equation}\label{stima3}
\begin{aligned}
&\left\vert \Big(H_p(x,m,Du)-H_p(x,m^{n-1},Du^{n-1})\Big)(Du^n-Du)+\Big(H_m(x,m,Du)-H_m(x,m^{n-1},Du^{n-1})\Big) (m^n-m)\right\vert \\
  &\leq\vert H_p(x,m,Du)-H_p(x,m^{n-1},Du^{n-1})\vert \vert Du^n-Du\vert+\vert H_m(x,m,Du)-H_m(x,m^{n-1},Du^{n-1})\vert \vert m^n-m\vert\\
  &\leq \Big(C(1+\vert Dv^{n-1}\vert)\vert \rho^{n-1}\vert +C\vert Dv^{n-1}\vert \Big)\vert Dv^n\vert
+\Big(C(1+\vert Dv^{n-1}\vert^2)\vert \rho^{n-1}\vert \vert +C(1+\vert Dv^{n-1}\vert)\vert Dv^{n-1}\vert \Big)\vert \rho^n\vert\\
 &\leq C\Big(\vert Dv^{n-1}\vert +\vert Dv^{n-1}\vert^2+\vert Dv^{n-1}\vert \vert \rho^{n-1}\vert +\vert Dv^{n-1}\vert^2 \vert \rho^{n-1}\vert +\vert \rho^{n-1}\vert \Big)(\vert Dv^n\vert+\vert \rho^n\vert)\\
 &\leq C\Big(\vert Dv^{n-1}\vert +\vert Dv^{n-1}\vert^2+\vert Dv^{n-1}\vert^4+\vert \rho^{n-1}\vert+\vert \rho^{n-1}\vert^2\Big)\big(\vert Dv^n\vert+\vert \rho^n\vert \big).
\end{aligned}
\end{equation}
To estimate the terms containing $f$ in \cref{stima_a}, we observe that
\begin{equation*}
\begin{aligned}
&-f'(m)(m^n-m)+f(m^{n-1})-f(m)+f'(m^{n-1})(m^n-m^{n-1})\\
={}&f(m^{n-1})-f(m)-f'(m)(m^{n-1}-m)+ (f'(m^{n-1})-f'(m))(m^n-m^{n-1}).
\end{aligned}
\end{equation*}

Exploiting that $f'(\cdot)$ is globally Lipschitz, see (A3), we have
\begin{equation}\label{uniform bound f''}
\vert f(m^{n-1})-f(m)-f'(m)(m^{n-1}-m)\vert \leq C|m^{n-1}-m|^2,
\end{equation} 
and
\begin{equation}\label{est_f'}
	\begin{aligned}
&\left\vert \Big(f'(m^{n-1})-f'(m)\Big)(m^n-m^{n-1})\right\vert \\
&\le \vert f'(m^{n-1})-f'(m)\vert \vert m^n-m\vert+\vert f'(m^{n-1})-f'(m)\vert \vert m-m^{n-1}\vert \\
&\le C\Big(|\rho^n||\rho^{n-1}|+|\rho^{n-1}|^2\Big).
\end{aligned}  
\end{equation}

Finally, by \cref{stima2}, \cref{stima3}, \cref{uniform bound f''} and \cref{est_f'}, we get 
\begin{equation}\label{a bound}
\begin{aligned}
\|a\|_{\mathcal{C}^{0}} \leq {}&C\Big(\|v^{n-1}\|_{\mathcal{C}^{1,0}}^2+\| \rho^{n-1}\|_{\mathcal{C}^{0}}^2+\|v^{n-1}\|_{\mathcal{C}^{1,0}}^4+\| \rho^{n-1}\|_{\mathcal{C}^{0}}^4\\
&+\big(\|v^{n-1}\|_{\mathcal{C}^{1,0}} +\|v^{n-1}\|_{\mathcal{C}^{1,0}}^2+\|v^{n-1}\|_{\mathcal{C}^{1,0}}^4+\|\rho^{n-1}\|_{\mathcal{C}^{0}}+\|\rho^{n-1}\|_{\mathcal{C}^{0}}^2\big)\big(\|v^{n}\|_{\mathcal{C}^{1,0}}+ \|\rho^{n}\|_{\mathcal{C}^{0}}\big)\Big).
\end{aligned}
\end{equation}
Now we consider the equation satisfied by $\rho^n$. We have 
$$
\partial_t \rho^n -\Delta \rho^n-{\rm{div}}(\rho^nH_p(x,m,Du))-{\rm{div}}(m\rho^nH_{pm}(x,m,Du))=
{\rm{div}}\big(mH_{pp}(x,m,Du)Dv^n\big)+{\rm{div}}(b),
$$
with
\begin{equation*}
\begin{aligned}
b:={}&-\rho^nH_p(x,m,Du)-m\rho^nH_{pm}(x,m,Du)-mH_{pp}(x,m,Du)Dv^n\\
&+ m^{n-1}H_p(x,m^{n-1},Du^{n-1}) -mH_p(x,m,Du)             \\
&+ (m^{n}- m^{n-1}) H_p(x,m^{n-1},Du^{n-1})+m^{n-1} H_{pm}(x,m^{n-1},Du^{n-1})(m^{n}- m^{n-1}) \\
&+m^{n-1}H_{pp}(x,m^{n-1},Du^{n-1})(Du^n-Du^{n-1}).
\end{aligned}
\end{equation*}
To estimate $\|b\|_{\mathcal{C}^0}$, we  start observing that 
\begin{equation}\label{stima_b}
\begin{aligned}
b={}&m^{n-1}H_p(x,m^{n-1},Du^{n-1}) -mH_p(x,m,Du)-m^{n-1}H_{pp}(x,m^{n-1},Du^{n-1})(Du^{n-1}-Du)\\
&-(m^{n-1}-m) H_p(x,m^{n-1},Du^{n-1})-m^{n-1} H_{pm}(x,m^{n-1},Du^{n-1})(m^{n-1}-m) \\
&+\rho^n\Big(H_p(x,m^{n-1},Du^{n-1})-H_p(x,m,Du)\Big)+\rho^n\Big(m^{n-1} H_{pm}(x,m^{n-1},Du^{n-1})-mH_{pm}(x,m,Du)\Big)\\
&+Dv^n\Big(m^{n-1}H_{pp}(x,m^{n-1},Du^{n-1})-mH_{pp}(x,m,Du)\Big).
\end{aligned}
\end{equation}
Denoting $\Phi(x,m,p)=mH_p(x,m,p)$, then by elementary calculation
\begin{align*}
&\partial_{p}\Phi(x,m,p)=mH_{pp}(x,m,p),\,\partial_{pp}\Phi(x,m,p)=mH_{ppp}(x,m,p),\\
&\partial_{m}\Phi(x,m,p)=H_p(x,m,p)+mH_{pm}(x,m,p),\,\partial_{mm}\Phi(x,m,p)=2H_{pm}(x,m,p)+mH_{mmp}(x,m,p).
\end{align*}
It is clear from (A2) that 
\begin{equation*}
\begin{aligned}
\sup_\tau \left\vert \big(m+\tau(m^{n-1}-m)\big)H_{ppp}(x,m+\tau(m^{n-1}-m),Du+\tau(Du^{n-1}-Du))\right\vert \leq C\big(\vert m\vert+|\rho^{n-1}|\big).
\end{aligned}
\end{equation*}
From \cref{Hmp} and 
\begin{equation}\label{Hmmp}
\sup_\tau \big\vert H_{mmp}(x,m+\tau(m^{n-1}-m),Du+\tau(Du^{n-1}-Du))\big\vert \leq C\big(\vert Du\vert+\vert Du^{n-1}-Du\vert \big),
\end{equation}
we  obtain
\begin{equation*}
\begin{aligned}
&\sup_\tau \left\vert 2H_{pm}(x,m+\tau(m^{n-1}-m),Du+\tau(Du^{n-1}-Du))\right\vert \\
&+\sup_\tau \left\vert  \big(m+\tau(m^{n-1}-m)\big)H_{pmm}(x,m+\tau(m^{n-1}-m),Du+\tau(Du^{n-1}-Du))\right\vert\\
\leq {}& C\big(1+\vert Dv^{n-1}\vert+\vert \rho^{n-1}\vert+\vert Dv^{n-1}\vert \vert \rho^{n-1}\vert\big).
\end{aligned}
\end{equation*}
In addition, we have
\begin{equation*}
\begin{aligned}
&\sup_\tau \Big\vert H_{pp}\big(x,m+\tau(m^{n-1}-m),Du+\tau(Du^{n-1}-Du)\big)\\
&+\big(m+\tau(m^{n-1}-m)\big)H_{ppm}\big(x,m+\tau(m^{n-1}-m),Du+\tau(Du^{n-1}-Du)\big)\Big\vert\\
\leq {}& C(1+\vert \rho^{n-1}\vert).
\end{aligned}
\end{equation*}
Collecting these estimates we   obtain
\begin{equation}\label{stima5}
	\begin{aligned}
&\Big\vert m^{n-1}H_p(x,m^{n-1},Du^{n-1}) -mH_p(x,m,Du)-m^{n-1}H_{pp}(x,m^{n-1},Du^{n-1})(Du^{n-1}-Du)\\
&-(m^{n-1}-m) H_p(x,m^{n-1},Du^{n-1})-m^{n-1} H_{pm}(x,m^{n-1},Du^{n-1})(m^{n-1}-m) \Big\vert \\
  &\leq C\big(1+|\rho^{n-1}|\big)\vert Dv^{n-1}\vert^2+C\big(1+\vert v^{n-1}\vert+\vert \rho^{n-1}\vert+\vert Dv^{n-1}\vert \vert \rho^{n-1}\vert\big)\vert \rho^{n-1}\vert^2+C(1+\vert \rho^{n-1}\vert)\vert Dv^{n-1}\vert \vert \rho^{n-1}\vert\\
  &\leq C\big(\vert Dv^{n-1}\vert^2+\vert \rho^{n-1}\vert^2+\vert \rho^{n-1}\vert^3+\vert Dv^{n-1}\vert \vert \rho^{n-1}\vert+\vert Dv^{n-1}\vert^2 \vert \rho^{n-1}\vert+\vert Dv^{n-1}\vert \vert \rho^{n-1}\vert^2+\vert Dv^{n-1}\vert \vert \rho^{n-1}\vert^3\big)\\
  &\leq C\big(\vert Dv^{n-1}\vert^2+\vert \rho^{n-1}\vert^2+\vert Dv^{n-1}\vert^4+\vert \rho^{n-1}\vert^4+\vert \rho^{n-1}\vert^3+\vert \rho^{n-1}\vert^6\big).
\end{aligned}
\end{equation}
Moreover,   from \cref{H}, it follows that
\begin{equation}\label{stima6}
\left\vert H_p(x,m^{n-1},Du^{n-1})-H_p(x,m,Du)\right\vert \leq C\left(\vert Dv^{n-1}\vert+(1+\vert Dv^{n-1}\vert)\vert \rho^{n-1}\vert\right),
\end{equation}
and, from \cref{H} and \cref{Hmmp},
\begin{equation}\label{stima7}
\begin{aligned}
&\left\vert m^{n-1} H_{pm}(x,m^{n-1},Du^{n-1})-mH_{pm}(x,m,Du)\right\vert \\
\leq {}&C(1+\vert \rho^{n-1}\vert)\vert Dv^{n-1}\vert+C\Big(1+\vert v^{n-1}\vert+(1+\vert v^{n-1}\vert )(1+\vert \rho^{n-1}\vert)\Big)\vert \rho^{n-1}\vert\\
\leq {}& C\big(\vert Dv^{n-1}\vert+\vert \rho^{n-1}\vert+\vert Dv^{n-1}\vert^2+\vert \rho^{n-1}\vert^2+\vert \rho^{n-1}\vert^4\big),
\end{aligned}
\end{equation}
\begin{equation}\label{stima8}
\begin{aligned}
\left\vert m^{n-1}H_{pp}(x,m^{n-1},Du^{n-1})-mH_{pp}(x,m,Du)\right\vert \leq C(1+\vert \rho^{n-1}\vert)\vert Dv^{n-1}\vert+C\vert \rho^{n-1}\vert+C(1+\vert \rho^{n-1}\vert)\vert \rho^{n-1}\vert.
\end{aligned}
\end{equation}
Replacing estimates  \cref{stima5}--\cref{stima8} in \cref{stima_b},   we get
\begin{equation}\label{b bound}
\begin{aligned}
&\|b\|_{\mathcal{C}^{0}} 
\leq  C\Big(\|v^{n-1}\|^2_{\mathcal{C}^{1,0}}+\|v^{n-1}\|^4_{\mathcal{C}^{1,0}}+\|\rho^{n-1}\|^2_{\mathcal{C}^{0}}+\|\rho^{n-1}\|^3_{\mathcal{C}^{0}}+\|\rho^{n-1}\|^4_{\mathcal{C}^{0}}+\|\rho^{n-1}\|^6_{\mathcal{C}^{0}}\\
&+\big(\|v^{n-1}\|_{\mathcal{C}^{1,0}}+ \|\rho^{n-1}\|_{\mathcal{C}^{0}}+\|v^{n-1}\|^2_{\mathcal{C}^{1,0}}+ \|\rho^{n-1}\|^2_{\mathcal{C}^{0}}+ \|\rho^{n-1}\|^4_{\mathcal{C}^{0}}\big)\big(\|v^{n}\|_{\mathcal{C}^{1,0}}+ \|\rho^{n}\|_{\mathcal{C}^{0}}\big)\Big).
\end{aligned}
\end{equation}
Finally,  from \cref{v rho estimate} and estimates \cref{a bound} and \cref{b bound}, we have 
\begin{equation}\label{v+rho}
\begin{aligned}
&\|v^n\|_{\mathcal{C}^{1,0}}+\|\rho^n\|_{\mathcal{C}^{0}}
\leq C(\|a\|_{\mathcal{C}^{0}}+\|b\|_{\mathcal{C}^{0}})\\
&\leq K\Big[\|v^{n-1}\|^2_{\mathcal{C}^{1,0}}+\|v^{n-1}\|^4_{\mathcal{C}^{1,0}}+\|\rho^{n-1}\|^2_{\mathcal{C}^{0}}+\|\rho^{n-1}\|^3_{\mathcal{C}^{0}}+\|\rho^{n-1}\|^4_{\mathcal{C}^{0}}+\|\rho^{n-1}\|^6_{\mathcal{C}^{0}}\\
&+\big(\|v^{n-1}\|_{\mathcal{C}^{1,0}}+ \|\rho^{n-1}\|_{\mathcal{C}^{0}}+\|v^{n-1}\|^2_{\mathcal{C}^{1,0}}+\|v^{n-1}\|^4_{\mathcal{C}^{1,0}}+ \|\rho^{n-1}\|^2_{\mathcal{C}^{0}}+ \|\rho^{n-1}\|^4_{\mathcal{C}^{0}}\big)\big(\|v^{n}\|_{\mathcal{C}^{1,0}}+\|\rho^{n}\|_{\mathcal{C}^{0}}\big)\Big],
\end{aligned}
\end{equation}
where $K$ is a constant which depends only on the data of the problem. Without loss of generality, we assume that $K>1$. 
Assume that  initial guess $(u^0,m^0)$ of the Newton method satisfies
$$
\|v^{0}\|_{\mathcal{C}^{1,0}}+\|\rho^{0}\|_{\mathcal{C}^{0}}\leq \frac{1}{12K}, 
$$
where $K$ as in \cref{v+rho}. Since $K>1$, we have 
$$
\|v^0\|^4_{\mathcal{C}^{1,0}}\leq \|v^{0}\|_{\mathcal{C}^{1,0}},\,\|\rho^{0}\|^6_{\mathcal{C}^{0}}\leq  \|\rho^{0}\|_{\mathcal{C}^{0}},
$$
\begin{equation}\label{stima9}	
\|v^{0}\|_{\mathcal{C}^{1,0}}+ \|\rho^{0}\|_{\mathcal{C}^{0}}+\|v^{0}\|^2_{\mathcal{C}^{1,0}}+\|v^{0}\|^4_{\mathcal{C}^{1,0}}+ \|\rho^{0}\|^2_{\mathcal{C}^{0}}+ \|\rho^{0}\|^4_{\mathcal{C}^{0}}\leq 3(\|v^{0}\|_{\mathcal{C}^{1,0}}+\|\rho^{0}\|_{\mathcal{C}^{0}})\leq \frac{1}{4K},
\end{equation}
$$
\|v^{0}\|^2_{\mathcal{C}^{1,0}}+\|\rho^{0}\|^2_{\mathcal{C}^{0}}\leq 2(\|v^{0}\|_{\mathcal{C}^{1,0}}+\|\rho^{0}\|_{\mathcal{C}^{0}})^2\leq \frac{1}{72K},
$$
$$
\|v^{0}\|^2_{\mathcal{C}^{1,0}}+\|v^{0}\|^4_{\mathcal{C}^{1,0}}+\|\rho^{0}\|^2_{\mathcal{C}^{0}}+\|\rho^{0}\|^3_{\mathcal{C}^{0}}+\|\rho^{0}\|^4_{\mathcal{C}^{0}}+\|\rho^{0}\|^6_{\mathcal{C}^{0}}\leq 4(\|v^{0}\|^2_{\mathcal{C}^{1,0}}+\|\rho^{0}\|^2_{\mathcal{C}^{0}})\leq \frac{1}{18K}.
$$
Replacing the previous estimates in \cref{v+rho} for $n=1$, we get
\begin{equation*}
\|v^1\|_{\mathcal{C}^{1,0}}+\|\rho^1\|_{\mathcal{C}^{0}}\leq \frac{1}{18K}+\frac{1}{4}(\|v^1\|_{\mathcal{C}^{1,0}}+\|\rho^1\|_{\mathcal{C}^{0}}).
\end{equation*}
Hence, by absorbing the term $\frac{1}{4}(\|v^1\|_{\mathcal{C}^{1,0}}+\|\rho^1\|_{\mathcal{C}^{0}})$ on the right hand side, we get
\begin{equation}\label{est2}
\|v^1\|_{\mathcal{C}^{1,0}}+\|\rho^1\|_{\mathcal{C}^{0}}\leq  \frac{2}{27K}<\frac{1}{12K}.
\end{equation}
Arguing  iteratively, we have that, if $\|v^{0}\|_{\mathcal{C}^{1,0}}+\|\rho^{0}\|_{\mathcal{C}^{0}}\leq \frac{1}{12K}$, then 
\begin{equation*}
\|v^{n}\|_{\mathcal{C}^{1,0}}+\|\rho^{n}\|_{\mathcal{C}^{0}}\leq \frac{1}{12K}, 
\qquad\text{for any $n\in\N$.}
\end{equation*}
Repeating an estimate similar to \cref{stima9} for  $n-1$, we have 
$$
\|v^{n-1}\|_{\mathcal{C}^{1,0}}+ \|\rho^{n-1}\|_{\mathcal{C}^{0}}+\|v^{n-1}\|^2_{\mathcal{C}^{1,0}}+\|v^{n-1}\|^4_{\mathcal{C}^{1,0}}+ \|\rho^{n-1}\|^2_{\mathcal{C}^{0}}+ \|\rho^{n-1}\|^4_{\mathcal{C}^{0}}\leq \frac{1}{4K}.
$$
From \cref{v+rho} we obtain
$$
\|v^n\|_{\mathcal{C}^{1,0}}+\|\rho^n\|_{\mathcal{C}^{0}}\leq K\Big[8\big(\|v^{n-1}\|_{\mathcal{C}^{1,0}}+\|\rho^{n-1}\|_{\mathcal{C}^{0}}\big)^2+\frac{1}{4K}\big(\|v^{n}\|_{\mathcal{C}^{1,0}}+\|\rho^{n}\|_{\mathcal{C}^{0}}\big)\Big],
$$
hence we get the estimate
\begin{equation*} 
\|v^{n}\|_{\mathcal{C}^{1,0}}+\|\rho^{n}\|_{\mathcal{C}^{0}}\leq \frac{32K}{3}\big(\|v^{n-1}\|_{\mathcal{C}^{1,0}}+\|\rho^{n-1}\|_{\mathcal{C}^{0}}\big)^2. 
\end{equation*}
Multiplying both the side of the previous estimate for $\frac{32K}{3}$, we have
	$$
	\frac{32K}{3}\big(\|v^{n}\|_{\mathcal{C}^{1,0}}+\|\rho^{n}\|_{\mathcal{C}^{0}}\big)\leq \Big(\frac{32K}{3}\big(\|v^{n-1}\|_{\mathcal{C}^{1,0}}+\|\rho^{n-1}\|_{\mathcal{C}^{0}}\big)\Big)^2.
	$$
	Hence, as we have assumed $\big(\|v^{0}\|_{\mathcal{C}^{1,0}}+\|\rho^{0}\|_{\mathcal{C}^{0}}\big)\leq \frac{1}{12K}$, we obtain by induction that 
	$$
	\|v^{n}\|_{\mathcal{C}^{1,0}}+\|\rho^{n}\|_{\mathcal{C}^{0}}\leq \frac{3}{32K}\Big(\frac{8}{9}\Big)^{2^n}.
	$$
\end{proof} 
\begin{remark}\label{rem_classical}
 We can extend our main convergence rate result to MFG system with superquadratic Hamiltonians, if we assume the system admits a classical solution.  In this case, for the rate of convergence in \cref{Main Thm local}, it  is sufficient to assume that $H$ is smooth, without the uniform bounds on the derivatives in \cref{H}. Indeed, a careful inspection of the previous proof shows that the constant $K$ in   \cref{v+rho} depends  on $(u,m)$, the data of the problem and, in particular, on the derivative of the Hamiltonian computed in $Dv^{n-1}$, $\rho^{n-1}$. If we assume that  $\|v^0\|_{\mathcal{C}^{1,0}}+\|\rho^0\|_{\mathcal{C}^{0}}<1/12K$, then, arguing as in the proof, we have that $\|v^n\|_{\mathcal{C}^{1,0}}+\|\rho^n\|_{\mathcal{C}^{0}}<1/12K$ for any $n\in\N$. Hence, the constant   $K$ in \cref{v+rho}, which depends on $\|v^{n-1}\|_{\mathcal{C}^{1,0}}+\|\rho^{n-1}\|_{\mathcal{C}^{0}}$, does not change with the iterations. The restriction of quadratic $H$ in \cref{Main Thm local} is made to prove  existence and uniqueness of a classical solution to the MFG system. 
\end{remark}

	\begin{remark} \label{rem_f} Assumption (A3) requires the uniform Lipschitz continuity of $f'$ and therefore it excludes some interesting cases such as $f(m)=m^\alpha$, $\alpha\neq 1$. We show that we can at least consider the case  $\alpha\geq 2$.  The main difference from using (A3) is in the estimate   \cref{est_f'}. We replace the argument in the proof with
		\begin{align*}
			\left\vert f'(m^{n-1})-f'(m)\right\vert 
			&\leq (\alpha-1)(|m^{\alpha-2}|+|(m^{n-1})^{\alpha-2}|)|\rho^{n-1}|\\
			&\le  (\alpha-1)(2|m^{\alpha-2}|+|(\rho^{n-1})^{\alpha-2}|)|\rho^{n-1}|   
		\end{align*}
and therefore
		\begin{align*}
			&\left\vert \Big(f'(m^{n-1})-f'(m)\Big)(m^n-m^{n-1})\right\vert \\
			\leq  {}&(\alpha-1)(2|m^{\alpha-2}|+|(\rho^{n-1})^{\alpha-2}|)|\rho^{n-1}|^2+ (\alpha-1)(2|m^{\alpha-2}|+|(\rho^{n-1})^{\alpha-1}|)|\rho^{n-1}||\rho^n|.  
		\end{align*}
		Then one can proceed similarly as in \cref{Main Thm local}. So far we do not have a corresponding result for $0<\alpha<2$, $\alpha\neq 1$. For global (in time) solutions  to MFGs with separable Hamiltonians and $f(m)=m^\alpha$, $\alpha>0$ we refer to the paper of Cirant and Goffi \cite[Theorem 1.4]{cirant2021maximal}. 
	\end{remark}


\section{The Newton  method for the Mean Field Games system with  saparable Hamiltonian and  nonlocal coupling}\label{sec:newton_nonlocal}
In this section, we consider  the MFG system with Hamiltonian independent of $m$ and nonlocal coupling 
\be\label{MFG nonlocal}
\left\{\begin{aligned}
(i)\qquad &-\partial_t u -\Delta u+H(x,Du)=f[m](x) \qquad &&{\rm in}\,\,Q,\\
(ii) \qquad & \partial_t m-  \Delta m  - \text{div} \big( mH_p(x,Du)\big) =0 &&{\rm in}\,\,Q, \\
&m(x,0)=m_0(x) , \; u(x,T)=g[m(T)](x)&&{\rm in}\,\,\T^d.
\end{aligned}\right.
\ee
We assume that $m_0$ is as in (A1), the Hamiltonian $H$ satisfies (A2), while the assumption on $u_T$ in (A1) and  (A3) are replaced by
\begin{itemize}
\item[(A3')] $f,g : \states \times \mes(\states) \to \R$. $f$, $g$ and their space derivatives $\partial_{x_i}f$, $\partial_{x_i}g$, $\partial_{x_ix_j}g$ are all globally Lipschitz continuous. The measure derivatives $\frac{\delta f}{\delta m}$ and $\frac{\delta g}{\delta m}$$:\T^d\times \mes (\T^d)\times \T^d\rightarrow \R$ are also Lipschitz continuous. For any $m,m'\in \mes(\states)$, 
\begin{equation}\label{mono}
\begin{aligned}
\int_{\T^d} \left(f[m](x)-f[m'](x)\right)d(m-m')(x)\geq 0, \,\,\int_{\T^d} \left(g[m](x)-g[m'](x)\right)d(m-m')(x)\geq 0.
\end{aligned}
\end{equation}
\end{itemize}
\begin{remark}
\cref{mono} implies that $\frac{\delta f}{\delta m}$ and $\frac{\delta g}{\delta m}$ satisfy the following monotonicity property (explained for $f$):
\begin{equation}\label{f mono}
\int_{\T^d}\int_{\T^d}\frac{\delta f}{\delta m}[m](x)(y)\rho(x)\rho(y)dxdy\geq 0
\end{equation}
for any centered measure $\rho$, c.f. \cite[p. 36]{cardaliaguet2019}. 
\end{remark}
\begin{remark}\label{f g Lip} From assumption (A3'), there exists a constant $C>0$,
\begin{align}
\sup_{x\in \T^d}\big|f[m'](x)-f[m](x)\big|+\sup_{x,y \in \T^d}\big|\frac{\delta f}{\delta m}[m'](x)(y)-\frac{\delta f}{\delta m}[m](x)(y)\big|&\leq C {\bf{d}}_1(m,m'),\label{f Lip}\\
\sup_{x\in\T^d}\big|f[m' ](x) - f[m](x) - \int_{\states} \frac{\delta f}{\delta m}[m](x)(y)d(m'-m)(y)\big| &\leq C{\bf{d}}^2_1(m,m') ,\label{deltaf Lip}\\
\sup_{x\in\T^d}\big|g[m' ](x) - g[m](x) - \int_{\states} \frac{\delta g}{\delta m}[m](x)(y)d(m'-m)(y)\big| &\leq C{\bf{d}}^2_1(m,m') .\label{delta g Lip}
\end{align}

\end{remark}
\begin{remark} For the simplicity we will be using the shortened notation, c.f. \cite[p. 60]{cardaliaguet2019}
$$
\int_{\T^d} \frac{\delta f}{\delta m}[m](x)(y)dm'(y)=\frac{\delta f}{\delta m}[m](x)m'
$$
for the duality bracket between $\frac{\delta f}{\delta m}[m]$ and $m'$ at $x$. 
\end{remark}

The next two lemmas  are proved in Cardaliaguet Briani \cite[Lemma 5.2]{bc}.  A similar result with different functional spaces is discussed in Cardaliaguet, Delarue, Lasry and Lions \cite[Lemma 3.3.1]{cardaliaguet2019}.
\begin{lemma}\label{stable nonlocal}
	Under assumptions (A1), (A2) and (A3'), let $(u,m)$ be a classical solution to the system \eqref{MFG nonlocal}. Then, the unique weak solution of the  system 
	\be\label{linearized nonlocal}
	\left\{\begin{aligned}
		(i) \qquad &-\partial_t v -\Delta v+H_p(x,Du)Dv=\frac{\delta f}{\delta m}[m(t)]\rho &&{\rm in}\,\,Q,\\
		(ii) \qquad & \partial_t \rho- \Delta \rho  - {\rm{div}} \big( \rho H_p(x,Du)\big)={\rm{div}}\big(mH_{pp}(x,Du)Dv\big) &&{\rm in}\,\,Q, \\
		& \rho(x,0)=0 , \; v(x,T)=\frac{\delta g}{\delta m}[m(T)]\rho(T) &&{\rm in}\,\,\T^d
	\end{aligned}\right.
	\ee
	is the trivial solution $(v,\rho)=(0,0)$. 
\end{lemma}

\begin{lemma}\label{v rho estimate nonlocal}\label{v rho nonlocal_lemma}
	Given $a\in \mathcal{C}^0(Q)$, $b\in \mathcal C^0(Q;\R^d)$. Let $(u,m)$ be a classical solution to the system \eqref{MFG nonlocal} and $(v,\rho)$ be a classical solution of the perturbed linear system
	\be\label{v rho nonlocal}
	\left\{\begin{aligned}
		(i)\qquad &-\partial_t v -\Delta v+H_p(x,Du)Dv=\frac{\delta f}{\delta m}[m(t)](x)\rho+a(x,t)\qquad &&{\rm in}\,\,Q,\\
		(ii)\qquad &\partial_t \rho- \Delta \rho  - {\rm{div}} \big( \rho H_p(x,Du)\big) ={\rm{div}}\big(mH_{pp}(x,Du)Dv\big)+{\rm{div}}(b(x,t))&&{\rm in}\,\,Q, \\
		& \rho(x,0)=0 , \; v(x,T)=\frac{\delta g}{\delta m}[m(T)](x)\rho(T)+c(x)&&{\rm in}\,\,\T^d.
	\end{aligned}\right.
	\ee
	Then, there exists a constant $C>0$ depending on the coefficients of the problem, such that
	\begin{equation*} 
		\|v\|_{\mathcal{C}^{1,0}}+\|\rho\|_{\mathcal{C}^{0}}\leq C\left(\|a\|_{\mathcal{C}^{0}}+\|b\|_{\mathcal{C}^{0}}+\|c\|_{\mathcal{C}^{0}}\right).
	\end{equation*}
\end{lemma}
Existence and uniqueness result for a classical solution to \cref{MFG nonlocal} under rather general assumptions which, in particular, include (A1), (A2) and (A3'),  can be found in \cite{ll}.\par
  The Newton system for solving \cref{MFG nonlocal}, analogous to \cref{Newton system}, can be written as
\be\label{Newton system nonlocal}
\left\{\begin{aligned}
(i)\qquad &-\partial_t u^n -\Delta u^n+H_p(x,Du^{n-1})D(u^n-u^{n-1})\\
={}&-H(x,Du^{n-1})+f[m^{n-1}(t)](x)+\frac{\delta f}{\delta m}[m^{n-1}(t)](x)(m^n-m^{n-1})\qquad &&{\rm in}\,\,Q,\\
(ii)\qquad &\partial_t m^n- \Delta m^n  - \text{div} \big( m^nH_p(x,Du^{n-1})\big) ={\rm{div}}\big(m^{n-1}H_{pp}(x,Du^{n-1})(Du^n-Du^{n-1})\big)&&{\rm in}\,\,Q, \\
&m^n(x,0)=m_0(x) , \; u^n(x,T)=g[m^{n-1}(T)]+\frac{\delta g}{\delta m}[m^{n-1}(T)](x)\big(m^n(T)-m^{n-1}(T)\big)&&{\rm in}\,\,\T^d.
\end{aligned}\right.
\ee
Existence and uniqueness of a classical solution to \cref{Newton system nonlocal} can be proved as in   \cref{system n}.
\begin{theorem}
Let $(u,m)$ be  the solution of system \eqref{MFG nonlocal} and $(u^n,m^n)$ is the sequence generated by Newton's algorithm \eqref{Newton system nonlocal}. Set $v^n=u^n-u$, $\rho^n=m^n-m$. There exists a constant $\eta>0$ such that if $\|v^0\|_{\mathcal{C}^{1,0}}+\|\rho^0\|_{\mathcal{C}^{0}}\leq \eta$ then $\|v^n\|_{\mathcal{C}^{1,0}}+\|\rho^n\|_{\mathcal{C}^{0}}\rightarrow 0$ with  a quadratic rate of convergence.
\end{theorem}
\begin{proof}
We first  observe that  $v^n=u^n-u$, $\rho^n=m^n-m$ satisfy
\begin{align*}
&	-\partial_t v^n -\Delta v^n+H_p(x,Du)\cdot Dv^n=\frac{\delta f}{\delta m}[m](x)(\rho^n)+a,\\
&\partial_t \rho^n -\Delta \rho^n-{\rm{div}}(\rho^nH_p(x,Du))=
{\rm{div}}\big(mH_{pp}(x,Du)Dv^n\big)+{\rm{div}}(b),\\
&\rho^n(x,0)=0,\,v^n(x,T)= \frac{\delta g}{\delta m}[m(T)](x)(\rho^n)  +c(x),             
\end{align*}
where
\begin{equation}\label{nonlocal_a}
\begin{aligned}
a:={}& H_p(x,Du)(Du^n-Du)+H(x,Du)-H(x,Du^{n-1})-H_p(x,Du^{n-1})D(u^{n}- u^{n-1})\\
&-\frac{\delta f}{\delta m}[m](x)(m^n-m)+f(m^{n-1})-f(m)+\frac{\delta f}{\delta m}[m^{n-1}](x)(m^n-m^{n-1}),
\end{aligned}
\end{equation}
\begin{equation*}
\begin{aligned}
b:={}&-\rho^nH_p(x,Du)-mH_{pp}(x,Du)Dv^n+ m^{n-1}H_p(x,Du^{n-1}) -mH_p(x,Du)             \\
&+ (m^{n}- m^{n-1}) H_p(x,Du^{n-1})+m^{n-1}H_{pp}(x,Du^{n-1})(Du^n-Du^{n-1}),
\end{aligned}
\end{equation*}
$$
c:= g[m^{n-1}(T)](x)-g[m(T)](x)-\frac{\delta g}{\delta m}[m(T)](x)(\rho^n)+\frac{\delta g}{\delta m}[m^{n-1}(T)](x)\big(m^n(T)-m^{n-1}(T)\big).
$$
We first consider the nonlocal coupling terms as they constitute the main differences with respect to the proof of \cref{Main Thm local}. Rewrite the terms containing $f$ in
\cref{nonlocal_a} as 
\begin{equation*}
\begin{aligned}
&-\frac{\delta f}{\delta m}[m](x)(m^n-m)+f(m^{n-1})-f(m)+\frac{\delta f}{\delta m}[m^{n-1}](x)(m^n-m^{n-1})\\
={}&f(m^{n-1})-f(m)-\frac{\delta f}{\delta m}[m^{n-1}](x)(\rho^{n-1})+\Big(\frac{\delta f}{\delta m}[m^{n-1}](x)-\frac{\delta f}{\delta m}[m](x)\Big)(\rho^n).
\end{aligned}
\end{equation*}
By Lipschitz continuity of $\frac{\delta f}{\delta m}$, i.e. \cref{f Lip} and \cref{deltaf Lip}, we get
$$
\|-\frac{\delta f}{\delta m}[m](\cdot)(m^n-m)+f(m^{n-1})-f(m)+\frac{\delta f}{\delta m}[m^{n-1}](\cdot)(m^n-m^{n-1})\|_{\mathcal{C}^{0}}\leq C(\|\rho^{n-1}\|_{\mathcal{C}^{0}}^2+\|\rho^{n-1}\|_{\mathcal{C}^{0}}\|\rho^n\|_{\mathcal{C}^{0}}).
$$
Similarly, rewriting 
$$
c= g[m^{n-1}(T)](x)-g[m(T)](x)-\frac{\delta g}{\delta m}[m^{n-1}(T)](x)(\rho^{n-1}(T))+\Big(\frac{\delta g}{\delta m}[m^{n-1}(T)](x)-\frac{\delta g}{\delta m}[m(T)](x)\Big)(\rho^n(T)),
$$
we obtain 
$$
\|c\|_{\mathcal{C}^{0}}\leq C\left(\|\rho^{n-1}\|_{\mathcal{C}^{0}}^2+\|\rho^{n-1}\|_{\mathcal{C}^{0}}\|\rho^n\|_{\mathcal{C}^{0}}\right).
$$
By a straightforward adaptation of the proof of  \cref{Main Thm local}, we estimate the other terms in  $a$ and $b$. Indeed,  we have
\begin{equation*}
\begin{aligned}
&\left\vert H_p(x,Du)(Du^n-Du)+H(x,Du)-H(x,Du^{n-1})-H_p(x,Du^{n-1})D(u^{n}- u^{n-1})\right\vert \\
\leq {}&C(\vert Dv^{n-1}\vert^2+\vert Dv^{n-1}\vert\vert Dv^{n}\vert),
\end{aligned}
\end{equation*}
hence 
$$
\|a\|_{\mathcal{C}^{0}}\leq C\left(\|v^{n-1}\|_{\mathcal{C}^{1,0}}^2+\|v^{n-1}\|_{\mathcal{C}^{1,0}}\|v^n\|_{\mathcal{C}^{1,0}}+\|\rho^{n-1}\|_{\mathcal{C}^{0}}^2+\|\rho^{n-1}\|_{\mathcal{C}^{0}}\|\rho^n\|_{\mathcal{C}^{0}}\right).
$$
Moreover, by 
\begin{equation*}
\begin{aligned}
b:={}&\rho^n(H_p(x,Du)-H_p(x,Du^{n-1}))+\rho^{n-1}H_{pp}(x,Du)Dv^n+ m^{n-1}H_p(x,Du^{n-1}) -mH_p(x,Du)             \\
&-(m^{n-1}-m) H_p(x,Du^{n-1})-m^{n-1}H_{pp}(x,Du^{n-1})(Du^{n-1}-Du),
\end{aligned}
\end{equation*}
we  obtain
\begin{equation*}
\begin{split}
&\|b\|_{\mathcal{C}^{0}}\\
\leq {}&C\left(\|v^{n-1}\|_{\mathcal{C}^{1,0}}\|\rho^{n}\|_{\mathcal{C}^{0}}+ \|v^{n}\|_{\mathcal{C}^{1,0}}\|\rho^{n-1}\|_{\mathcal{C}^{0}}+\|\rho^{n-1}\|^2_{\mathcal{C}^{0}}+\|v^{n-1}\|_{\mathcal{C}^{1,0}}\|\rho^{n-1}\|_{\mathcal{C}^{0}}+(1+\|\rho^{n-1}\|_{\mathcal{C}^{0}})\|v^{n-1}\|^2_{\mathcal{C}^{1,0}}\right).
\end{split}
\end{equation*}
Collecting the estimate of $a$, $b$ and $c$, by \cref{v rho nonlocal_lemma} we obtain
$$
\|v^n\|_{\mathcal{C}^{1,0}}+\|\rho^n\|_{\mathcal{C}^{0}} \leq C\left(\|v^{n-1}\|_{\mathcal{C}^{1,0}}^2+\|v^{n-1}\|_{\mathcal{C}^{1,0}}^4+\|\rho^{n-1}\|^2_{\mathcal{C}^{0}}+(\|v^{n-1}\|_{\mathcal{C}^{1,0}}+\|\rho^{n-1}\|_{\mathcal{C}^{0}})(\|v^{n}\|_{\mathcal{C}^{1,0}}+\|\rho^{n}\|_{\mathcal{C}^{0}})\right).
$$
We omit the rest of the proof as it is very similar to \cref{Main Thm local}.
\end{proof}

\begin{remark}
The monotonicity conditions \cref{mono}   guarantee uniqueness of the  solution to the    MFG system with nonlocal coupling. If we do not assume \cref{mono}, the result and proof methodology in this section can be   adapted to prove  local convergence to a stable solution of a potential MFG. Recall that, for a potential MFG,  a solution  $(u,m)$ is said to be stable if the only solution to the   linearized MFG system at $(u,m)$   is the trivial one (see \cite{bc}). In other words, instead of proving that $(v,\rho)=0$ as in \cref{stable nonlocal}, we use it as part of the definition of the stable solution. We plan to study this problem in the future.
\end{remark}

\begin{remark}
In this paper we separated the discussions on MFGs with local or nonlocal couplings for simplicity. One can easily replace in \cref{MFG} one or both of the terms $f$ and $g$ by nonlocal couplings and obtain similar results as \cref{Well posed} and \cref{Main Thm local}. One can also consider the nonseparable Hamiltonians with nonlocal congestions, e.g. replacing $m$ in $H(x,m,p)$ by a convolution with some kernels, as in \cite{achdou2015}. However, even though the existence of solution to such type of MFGs has been demonstrated in \cite{achdou2015}, it is not clear how one can apply the Hessian condition \cref{H Hessian} to show the uniqueness of a global (in time) classical solution and the stability property (\cref{v rho estimate}) in the nonlocal congestion case. We leave further developments in this direction for the future.
\end{remark}

\appendix
\counterwithin{theorem}{section}
\section{Appendix: some classical parabolic estimates results}\label{sec:appendix}
Consider the linear parabolic equation:
\begin{equation}\label{Linear Parabolic}
		\left\{\begin{aligned}
		 \qquad &	-\partial_t u-\Delta u+b(x,t)\cdot  \nabla u+c(x,t)u=f(x,t) \qquad &&{\rm in}\,\,Q,\\
		 \qquad &u(x,T)=u_T(x) &&{\rm in}\,\,\T^d.
		\end{aligned}\right.
	\end{equation}	
The following two results are very classical for equations on cylinders with boundary conditions (see \cite[Theorem 5.1, p. 320 ]{LSU} and \cite[Theorem 9.1 pp. 341-342]{LSU}). A complete proof of them in the flat torus case can be found in \cite[Appendix pp. 17-18]{cirant2020}.
\begin{proposition}\label{linear estim}
	Let $b\in \mathcal C^{\alpha,\alpha/2}(Q;\R^d)$, $c$ and $f$ belong to $\mathcal C^{\alpha,\alpha/2}(Q)$ and $u_T\in C^{2+\alpha}(\T^d)$.  Then the problem \eqref{Linear Parabolic}
admits a unique solution $u\in \mathcal C^{2+\alpha,1+\alpha/2}(Q)$  and it holds
	\begin{equation*}
		\| u\| _{\mathcal C^{2+\alpha,1+\alpha/2}(Q)}\leq C\left(\| f\| _{\mathcal C^{\alpha,\alpha/2}(Q)}+\| u_T\| _{\mathcal C^{2+\alpha}(\T^d)}\right),
	\end{equation*}
	where $C$ depends on the $\mathcal C^{\alpha,\alpha/2}$ norms of $b$, $c$ and remains bounded for bounded values of $T$.\end{proposition}	
		\begin{proposition}\label{linear estim Sobolev}
	Let $r>d+2$, $b\in L^\infty(Q;\R^d)$, $c\in L^\infty(Q)$, $f\in L^r(Q)$ and $u_T\in W^{2-\frac{2}{r}}_r(\T^d)$.  Then the problem \eqref{Linear Parabolic} admits a unique solution $u\in W^{1,2}_r(Q)$ and it holds
	\begin{equation*}
		\| u\| _{W^{2,1}_r(Q)}\leq C\left(\| f\| _{L^r(Q)}+\| u_T\| _{W^{2-\frac{2}{r}}_r(\T^d)}\right),
	\end{equation*}
	where $C$ depends on the norms of $b$, $c$ and remains bounded for bounded values of $T$. Moreover, the following embedding  holds:
	\begin{equation}\label{embedding}
\|u\|_{\mathcal C^{2-\frac{d+2}{r},1-\frac{d+2}{2r}}(Q)} \leq C\|u\|_{W^{2,1}_r(Q)}.
\end{equation}
(see \cite[Corollary pp. 342-343]{LSU})
\end{proposition}
	Next we introduce some results for the parabolic equations in divergence form, for the proof we refer to \cite[Proposition A.3]{tang2022learning}. We use the notation $\iota:=\{1,2\}$.
\begin{proposition}\label{m stability} Let $r>d+2$,  $m_0\in W^1_r(\T^d)$,
	 $q_\iota\in L^\infty(Q;\R^d)$ and consider the parabolic equation in divergence form 
\begin{equation}\label{m1}
		\left\{\begin{aligned}
 \qquad &	\partial_t m_\iota- \Delta m_\iota - {\rm{div}} ( m_\iota q_\iota) =0\qquad &&{\rm in}\,\,Q,\\
		\qquad &	m_\iota(0,x)=m_0(x)&&{\rm in}\,\,\T^d.
		\end{aligned}\right.
	\end{equation}
Then there exists a unique solution $m_\iota$ in $\mathcal H^1_r(Q)$ to \eqref{m1} Then, 
there exists a constant $C$ depending only on $r,T,d, R$, $\|m_0\|_{W^1_r(\T^d)}$ and $\|q_\iota\|_{L^\infty(Q;\R^d)}$  such that 
$$\| m_\iota\|_{\mathcal H^1_r(Q)}\leq C.$$ 
Moreover, denote by $\delta m=m_1-m_2$ and $\delta q=q_1-q_2$. Then  
$$\|\delta m\|_{\mathcal H^1_r(Q)}\leq C\|\delta q\|_{L^\infty(Q;\R^d)}$$	
with $C$ as above.
\end{proposition}

 We show the well posedness of the MFG system with non-separable Hamiltonian by means of an argument similar to \cite[Lemma 4.1]{achdou2018}.

\begin{proof}[Proof of   \cref{Well posed}]
Define $\mathbf{X}:=\left\{m\in \mathcal C^0(Q): m>0, m(x,0)=m_0(x), \int_{\T^d}m(x,t)dx=1\right\}$. Consider the compact mapping $\tilde{m}=\mathbf{T}(m): \mathbf{X}\rightarrow \mathcal{C}^{\alpha,\alpha/2}(Q)$ ,
\be\label{fp appendix}
\left\{\begin{aligned}
(i) \qquad &-\partial_t u -\Delta u+H(x,m,Du)=f(m) \qquad &&{\rm in}\,\,Q,\\
(ii) \qquad &\partial_t \tilde{m}-  \Delta \tilde{m}  - \text{div} \big( \tilde{m}H_p(x,m,Du)\big) =0&&{\rm in}\,\,Q, \\
&\tilde{m}(x,0)=m_0(x) , \; u(x,T)=u_T(x)&&{\rm in}\,\,\T^d.
\end{aligned}\right.
\ee
$f: \mathbf{X}\rightarrow \mathbb{R}$ is uniformly bounded. By (A2), $H(x,m,p)\leq \bar C (|p|^2 + 1)$. It follows from the standard result of quasilinear parabolic equation (see Theorem 6.3 in Chapter 5 of \cite{LSU}) that \cref{fp appendix} (i) has a unique solution $u\in L^\infty(0,T;W^1_{\infty}(\T^d))$ and $\|u\|_{L^\infty(0,T;W^1_{\infty}(\T^d))}$ is bounded independently of $m$ in $\mathbf{X}$. From the boundedness of $\|Du\|_{L^\infty(Q)}$ and (A2), $H_m(x,m,p)$ is bounded independently of $m$ in $\mathbf{X}$. Under the additional assumption that $f'(m)$ is bounded independently of $m$ in $\mathbf{X}$,  it is easy to see that the map $m\mapsto u$ from $\mathbf{X}$ to $L^\infty(0,T;W^1_{\infty}(\T^d))$ is continuous. \par
\cref{fp appendix} (ii) has a unique solution $\tilde{m}\in \mathbf{X}\cap \mathcal{C}^{\alpha,\alpha/2}(Q)$, which is bounded in $\mathcal{C}^{\alpha,\alpha/2}(Q)$ uniformly with respect to $m$ in $\mathbf{X}$, see e.g. \cite{LSU}, Chapter 3, Theorem 10.1. From \cref{m stability} and the continuous embedding of $\mathcal{H}_r^1(Q)$ onto $\mathcal{C}^{\alpha,\alpha/2}(Q)$, the map $H_p(x,m,Du) \mapsto \tilde{m}$ is continuous from $L^\infty(Q)$ to $\mathcal{C}^{\alpha,\alpha/2}(Q)$. From (A2) $H_p(x,m,Du)$ is continuous with respect to both $m$ and $Du$, independently of $m$ in $\mathbf{X}$. As we have shown the map $m\mapsto u$ from $\mathbf{X}$ to $L^\infty(0,T;W^1_{\infty}(\T^d))$ is continuous, hence the map $m\mapsto \tilde{m}$ is continuous from $\mathbf{X}$ to $\mathbf{X}\cap \mathcal{C}^{\alpha,\alpha/2}(Q)$. We can obtain by Schauder fixed point theorem that the map $\mathbf{T}$ has at least one fixed point.  Regularity of the solution follows from assumptions (A1), (A2), (A3) and previously cited results in \cite{LSU}.
Uniqueness   follows from (A2) (\eqref{H Hessian} in particular) and (A3). We refer to \cite[Ch.1, Theorem 13]{achdouCetraro} for details.
\end{proof}

%
%
%
%
%
%

\begin{flushright}
	\noindent \verb"fabio.camilli@uniroma1.it"\\
	SBAI, Sapienza Universit\`{a} di Roma\\
	Roma (Italy)	
\end{flushright}

\begin{flushright}
	\noindent \verb"tangqingthomas@gmail.com"\\
	China University of Geosciences (Wuhan)\\
	Wuhan (China)	
\end{flushright} 

\end{document}